\documentclass[final]{article}

\usepackage[utf8]{inputenc}
\usepackage{amsmath, amsfonts, amsthm}
\usepackage{graphicx}
\usepackage{hyperref}
\usepackage{booktabs}
\usepackage{enumerate}
\usepackage{color}
\usepackage[square,numbers,sort&compress]{natbib}
\usepackage{jabbrv}

\usepackage[normalem]{ulem}

\usepackage[notref, notcite]{showkeys}

\usepackage[vlined, algoruled]{algorithm2e}

\newcommand{\dxk}{\Delta{x}^k}
\newcommand{\dyk}{\Delta{\lambda}^k}
\newcommand{\dsk}{\Delta{z}^k}

\newcommand{\alphapk}{\alpha_P^k}
\newcommand{\alphadk}{\alpha_D^k}

\newcommand{\xysk}[1]{x^{#1}, {\lambda}^{#1}, z^{#1}}
\newcommand{\R}{\mathbb{R}}
\newcommand{\hopdm}{\texttt{HOPDM}}
\newcommand{\qnipm}{\texttt{qnHOPDM}}
\newcommand{\qnipmmc}{\texttt{qnHOPDM}-mc}

\newcommand{\qnv}{\bar x}

\newcommand{\notfeas}{{}$^\mathrm{\,a}$}
\newcommand{\notopti}{{\!}$^\mathrm{\,b}$}
\newcommand{\blanksp}{{}$^\mathrm{\,\hspace{4pt}}$}

\DeclareMathOperator{\diag}{diag}

\newtheorem{criterion}{Criterion}
\newtheorem{lemma}{Lemma}

\title{Quasi-Newton approaches to Interior Point Methods for quadratic
  problems}

\author{%
  J. Gondzio\footnote{School of Mathematics, University of Edinburgh,
    Edinburgh, EH9 3FD, Scotland, United Kingdom. Email:
    \url{J.Gondzio@ed.ac.uk}} \and%
  F. N. C. Sobral\footnote{Corresponding author. Department of
    Mathematics, State University of Maringá, Avenida Colombo, 5790,
    Paraná, Brazil, 87020-900. Phone: +55 44 30116211. E-mail:
    \url{fncsobral@uem.br}}%
}

\date{ \small \bf Technical Report ERGO 18-015, School of Mathematics,
  June 25, 2018}

\begin{document}

\maketitle

\begin{abstract}
  Interior Point Methods (IPM) rely on the Newton method for solving
  systems of nonlinear equations. Solving the linear systems which
  arise from this approach is the most computationally expensive task
  of an interior point iteration. If, due to problem's inner
  structure, there are special techniques for efficiently solving
  linear systems, IPMs enjoy fast convergence and are able to solve
  large scale optimization problems. It is tempting to try to replace
  the Newton method by quasi-Newton methods. Quasi-Newton approaches
  to IPMs either are built to approximate the Lagrangian function for
  nonlinear programming problems or provide an inexpensive
  preconditioner. In this work we study the impact of using
  quasi-Newton methods applied directly to the nonlinear system of
  equations for general quadratic programming problems. The cost of
  each iteration can be compared to the cost of computing correctors
  in a usual interior point iteration. Numerical experiments show that
  the new approach is able to reduce the overall number of matrix
  factorizations and is suitable for a matrix-free implementation.\\

  \noindent \textbf{Keywords}: Broyden Method, Quasi-Newton, Interior
  Point Methods, Matrix-free, Quadratic Programming Problems
\end{abstract}

\section{Introduction}
\label{intro}

Let us consider the following general quadratic programming problem
\begin{equation}
  \label{def:qp}
  \begin{array}{ll}
    \min & \frac{1}{2} x^T Q x + c^T x \\
    \mbox{s. t.} & A x = b \\
    & x \ge 0,
  \end{array}
\end{equation}
where $x, c \in \R^n$, $b \in \R^m$, $Q \in \R^{n \times n}$ and
$A \in \R^{m \times n}$. We will suppose that the rows of $A$ are
linearly independent. Define function $F: \R^{2n + m} \to \R^{2n + m}$
by
\begin{equation}
  \label{def:f}
  F(x, \lambda, z) =
  \begin{bmatrix}
    - Q x + A^T \lambda + z - c \\
    A x - b \\
    X Z e
  \end{bmatrix},
\end{equation}
where $X, Z \in \R^{n \times n}$ are diagonal matrices defined by
$X = \diag(x)$ and $Z = \diag(z)$, respectively, and $e$ is the vector
of ones of appropriate size. First order necessary conditions
for~\eqref{def:qp} state that, if $x^* \ge 0$ is a minimizer, then
there exist $z^* \in \R^n$, $z^* \ge 0$, and $\lambda^* \in \R^m$ such
that $F(\xysk{*}) = 0$.

Primal-Dual IPMs try to solve~\eqref{def:qp}
by solving a sequence of relaxed constrained nonlinear equations in
the form of
\begin{equation}
  \label{ipnls}
  F(x, y, s) =
  \begin{bmatrix}
    0 \\ 0 \\ \mu e
  \end{bmatrix},
  \quad x, s > 0,
\end{equation}
where $\mu \in \R$ is called the barrier parameter, which is
associated with the logarithmic barrier applied to the inequalities
$x \ge 0$ used to derive the method~\cite{Gondzio2012a,
  Wright1997}. As $\mu \to 0$ more importance is given to optimality
over feasibility. Systems of type~\eqref{ipnls} are not easy to
solve. When $\mu = 0$, they can be solved by general algorithms for
bounded nonlinear systems~\cite{Friedlander1997a, Kozakevich1996}. In
this case, a suitable merit function, usually $\|F(x)\|$, has to be
used to select the step-sizes. IPMs try to stay near the solution
of~\eqref{ipnls}, called the central path, and reduce $\mu$ at each
iteration. Instead of solving~\eqref{ipnls} exactly, one step of the
Newton method is applied. Thus, given an iterate $(\xysk{k})$, in the
interior of the bound constraints, i.e. $x^k, z^k > 0$, the next
point is given by
\begin{equation}
  \label{ipmss}
  (\xysk{k + 1}) = (\xysk{k}) + (\alpha_P \dxk, \alpha_D \dyk,
  \alpha_D \dsk),
\end{equation}
where $(\dxk, \dyk, \dsk)$ is computed by solving some Newton-like
systems
\begin{equation}
  \label{ipmit}
  J(\xysk{k})
  \begin{bmatrix}
    \dxk \\ \dyk \\ \dsk
  \end{bmatrix}
  = v,
\end{equation}
where $v \in \R^{2n + m}$ and
$J : \R^{2n + m} \to \R^{(2n + m) \times (2n + m)}$ is the Jacobian of
$F$, defined by
\begin{equation}
  \label{def:j}
  J(x, \lambda, z) = 
  \begin{bmatrix}
    -Q & A^T & I \\
    A  &   0 & 0 \\
    Z  &   0 & X
  \end{bmatrix}.
\end{equation}
Standard predictor-corrector algorithms solve~\eqref{ipmit} twice:
first the affine scaling predictor is computed for $v = - F(\xysk{k})$
and then the corrector step is computed using
$v = \begin{bmatrix} 0 & 0 & \sigma_k \mu_k e
\end{bmatrix}^T$,
with $\sigma_k \in (0, 1)$, $\mu_k = {x^k}^T z^k / n$.
Additional correctors can be computed in one iteration to further
accelerate convergence, such as second order
correctors~\cite{Mehrotra1992} or multiple centrality
correctors~\cite{Gondzio1996}. Scalars $\alpha_P$ and $\alpha_D$ are
selected such that $x^{k + 1} > 0$ and $s^{k + 1} > 0$, respectively.



The most expensive task during an interior point (IP) iteration is to
solve~\eqref{ipmit}.
The coefficient matrix $J(x, \lambda, z)$ is known as unreduced matrix
and has dimension $(2n + m) \times (2n + m)$, but its nice structure
allows efficient solution techniques to be used. The most common
approaches for solving the linear system in IPMs are to work with
augmented system or normal equations. If we eliminate $\Delta s$
in~\eqref{ipmit}, we have the augmented system for which we can solve
directly using matrix factorizations or compute adequate
preconditioners and solve iteratively by Krylov subspace methods. If
matrix $Q$ is easily invertible, or $Q = 0$ (linear programming
problems), it is possible to further eliminate $\Delta x$ and solve
the normal equations by Cholesky factorization or by Conjugate
Gradients, depending on the size of the problem. For both approaches
it is known that computing good preconditioners or computing the
factorization can be most expensive part of the
process. Therefore~\eqref{ipmit} can be solved several times for the
same $J(\xysk{k})$ with different right-hand sides, in a classical
predictor-corrector approach~\cite{Mehrotra1992} or in the multiple
centrality correctors framework~\cite{Colombo2008, Gondzio2012a}. In
this work we will extensively use the fact that the backsolves
in~\eqref{ipmit} are less expensive than computing a good
preconditioner or factorization.

Although $J(x, y, s)$ is unsymmetric, under reasonable assumptions
Greif, Moulding and Orban showed that it has only real
eigenvalues~\cite{Greif2014}. Based on those results, Morini,
Simoncini and Tani~\cite{Morini2017} developed preconditioners for the
unreduced matrix and compared the performance of interior point
methods using unreduced matrices and augmented system. The unreduced
matrix has also two more advantages, when compared to augmented system
and normal equations. First, small changes of variables $x$ or $z$
result in small changes in $J(x, \lambda, z)$. Second, $J$ is the
Jacobian of $F$, so it is possible to approximate it by building
models or evaluating $F$ on some extra points. These two
characteristics are explored in this work.

Since $J$ is the Jacobian of $F$, it is natural to ask if it can be
approximated by evaluating $F$ in some points. Function $F$ is
composed by two linear and one nonlinear functions. Therefore, the
only part of $J$ which may change during iterations is the third
row. Moreover, it can be efficiently stored by just storing $A$, $Q$,
$x$ and $z$. Since computing and storing $J$ is inexpensive, the only
reason to use an approximation $B$ of $J$ is if system~\eqref{ipmit},
using $B_k$ instead of $J(\xysk{k})$, becomes easier to solve. That is
where quasi-Newton methods and low rank updates become an interesting
tool in interior point methods.

Quasi-Newton methods are well known techniques for solving large scale
nonlinear systems or nonlinear optimization problems. The main
motivation is to replace the Jacobian used by the traditional Newton
method by its good and inexpensive approximation. Originally, they
were useful to avoid computing the derivatives of $F$, but they have
become popular as a large scale tool, since they usually do not need
to explicitly build matrices and enjoy superlinear
convergence. Classical references for quasi-Newton methods
are~\cite{DennisJr.1996, Martinez2000b} for nonlinear equations
and~\cite{Nocedal2006} for unconstrained optimization.

In the review~\cite{Martinez2000b} about practical quasi-Newton
methods for solving nonlinear equations, Martínez suggests that there
is room for studying such techniques in the interior point
context. The author points to the work of Dennis Jr., Morshedi and
Turner~\cite{Dennis1987} which applies quasi-Newton techniques to make
the projections in Karmarkar's algorithm cheaper. The authors write
the interpolation equations associated with the linear system in
interior point iterations and describe a fast algorithm to compute
updates and also to update an already existing Cholesky
factorization. When solving general nonlinear programming problems by
IPMs, a well known approach is to replace the Hessian of the
Lagrangian function by low rank approximations~\cite{Nocedal2006}.


In 2000, Morales and Nocedal~\cite{Morales2000a} used quasi-Newton
arguments to show that the directions calculated by the Conjugate
Gradient algorithm can be used to build an automatic preconditioner
for the matrix under consideration. The preconditioner is a sequence
of rank-one updates of an initial diagonal matrix. Such approach is
efficient when solving a sequence of linear systems with the same (or
a slowly varying) coefficient matrix. Based on those ideas, a limited
memory BFGS-like preconditioner for positive definite matrices was
developed in~\cite{Gratton2011} and was specialized for symmetric
indefinite matrices in~\cite{Gratton2016} . Recently, Bergamaschi
\emph{et al.}~\cite{Bergamaschi2018} developed limited-memory
BFGS-like preconditioners to KKT systems arising from IP iterations
and described their spectral properties. The approach was able to
reduce the number of iterations in the Conjugate Gradient algorithm,
but the approximation deteriorates as the number of interior point
iterations increase. Also, extra linear algebra has to be performed to
ensure orthogonality of the vectors used to build the updates.

In all works, with exception of~\cite{Dennis1987}, the main focus was
to use low rank updates of an already computed preconditioner such
that new preconditioners are constructed in an inexpensive way and
reduce the overall number of linear algebra iterations. In the present
work, our main objective is to work directly with nonlinear equations
and use low rank secant updates for computing the directions in the IP
iterations. We use least change secant updates, in particular Broyden
updates, and replace the Newton system~\eqref{ipmit} by an equivalent
one. Some properties of the method are presented and extensive
numerical experiments are performed. The main features of the proposed
approach are:
\begin{itemize}
\item Low rank approximations are matrix-free and use only vector
  multiplications and additions;

\item The quasi-Newton method for solving~\eqref{ipmit} can be easily
  inserted into an existing IPM;

\item The number of factorizations is reduced for small and large
  instances of linear and quadratic problems;

\item When the cost of the factorization is considerably higher than
  the cost of the backsolves, the total CPU time is also decreased.
\end{itemize}

In Section~\ref{back} we discuss the basic ideas of quasi-Newton
methods, in particular the Broyden method, which is extensively used
in the work. In Section~\ref{qn} we show that, if the initial
approximation is good enough, least change secant updates preserve
most of the structure of the true coefficient matrix and a traditional
IP iteration can be performed with the cost of computing correctors
only. New low rank secant updates, which are able to exploit the
sparsity of $J$ are also discussed. In Section~\ref{implementation} we
describe the aspects of a successful implementation of a quasi-Newton
interior point method. In Section~\ref{numerical} we compare our
approach with a research implementation of the primal-dual IPM for
solving small- and medium-sized linear and quadratic
problems. Finally, in Section~\ref{conclusions} we draw the
conclusions and mention possible extensions of the method.

\paragraph{Notation.} Throughout this work we use $F_k$ and
  $J_k$ as short versions of vector $F(\xysk{k})$ and matrix
  $J(\xysk{k})$, respectively. The vector $e$ denotes the vector of
  ones of appropriate dimension.

\section{Background for quasi-Newton methods}
\label{back}

Quasi-Newton methods can be described as algorithms which use
approximations to the Jacobian in the Newton method in order to solve
nonlinear systems. The approximations are generated using information
from previous iterations. Suppose that we want to find $\qnv \in \R^N$
such that $F(\qnv) = 0$, where $F : \R^N \to \R^N$ is continuously
differentiable. Given the current point $\qnv^k$ at iteration $k$,
Newton method builds a linear model of $F$ around $\qnv^k$ in order to
find $\qnv^{k + 1}$. Now, suppose that $\qnv^k$ and $\qnv^{k + 1}$
have already been calculated and let us create a linear model for $F$
around $\qnv^{k + 1}$:
\begin{equation}
  \label{linmodel}
  M_{k + 1}(\qnv) = F(\qnv^{k + 1}) + B_{k + 1} (\qnv - \qnv^{k + 1}).
\end{equation}
The choice $B_{k + 1} = J_{k + 1}$ results in the Newton method for
iteration $k + 1$. In secant methods, $B_{k + 1}$ is
constructed such that $M_{k + 1}$ interpolates $F$ at $\qnv^k$ and
$\qnv^{k + 1}$, which gives us the \emph{secant equation}
\begin{equation}
  \label{seceq}
  B_{k + 1} s_k = y_k,
\end{equation}
where $s_k = \qnv^{k + 1} - \qnv^k$ and
$y_k = F(\qnv^{k + 1}) - F(\qnv^k)$. When $s_k \ne 0$ and $N > 1$
there are more unknowns than equations and several choices for
$B_{k + 1}$ exist~\cite{DennisJr.1979, Martinez2000b}.

Let $B_k$ be the current approximation to $J_k$, the Jacobian of $F$
at $\qnv^k$ (it can be $J_k$ itself, for example). One of the most
often used simple secant approximations for unsymmetric Jacobians is
given by the Broyden ``good'' method. Given $B_k$, a new approximation
$B_{k + 1}$ to $J_{k + 1}$ is given by
\begin{equation}
  \label{gbroyd1}
  B_{k + 1} = B_k + \frac{(y_k - B_k s_k) s_k^T}{s_k^T s_k}.
\end{equation}
Matrix $B_{k + 1}$ is the closest matrix to $B_k$, in Frobenius norm,
which satisfies~\eqref{seceq}. The update of the Broyden method
belongs to the class of least change secant updates, since $B_{k + 1}$
is a rank-one update of $B_k$. As we are interested in solving a
linear system, it may be interesting to analyze matrix
$B_{k+1}^{-1} = H_{k + 1}$, which is obtained by the well known
Sherman-Morrison-Woodbury formula:
\begin{equation}
  \label{gbroyd2}
  H_{k + 1} = H_k + \frac{(s_k - H_k y_k) s_k^T H_k}{s_k^T H_k y_k}
  = \left( I + \frac{u_k s_k^T}{\rho_k} \right) H_k,
\end{equation}
where $u_k = s_k - H_k y_k$ and $\rho_k = s_k^T H_k y_k$. We can see
that $H_{k + 1}$ is also a least change secant update of $H_k$. To
store $H_{k + 1}$, one needs first to compute $H_k y_k$ and then store
one scalar and two vectors. Storing $u_k$ is more efficient than
storing $H_k s_k$ when $H_{k + 1}$ is going to be used more than
once. According to~\eqref{gbroyd2}, the cost of computing
$H_{k + 1} v$ is the cost of computing $H_k v$ plus one scalar product
and one sum of vectors times a scalar. After $\ell$ updates of an
initial approximation $B_{k - \ell}$, current approximation $H_{k}$ is
given by
\[
H_{k} = \left( I + \frac{u_{k - 1} s_{k - 1}^T}{\rho_{k - 1}} \right)
H_{k - 1} = \left[ \prod_{j = 1}^{\ell} \left( I + \frac{u_{k - j}
      s_{k - j}^T}{\rho_{k - j}} \right) \right] H_{k - \ell}.
\]

Instead of updating $B_k$ and then computing its inverse, the Broyden
``bad'' method directly computes the least change secant update of the
inverse:
\begin{equation}
  \label{bbroyd1}
H_{k + 1} = H_k + \frac{(s_k - H_k y_k) y_k^T}{y_k^T y_k} =
H_k V_k + \frac{s_k y_k^T}{\rho_k},
\end{equation}
where $V_k = \left( I - \frac{y_k y_k^T}{\rho_k} \right)$ and
$\rho_k = y_k^T y_k$. Similarly to $B_{k + 1}$ in~\eqref{gbroyd2},
$H_{k + 1}$ given by~\eqref{bbroyd1} is the closest matrix of $H_k$,
in the Frobenius norm, such that $H_{k + 1}^{-1}$
satisfies~\eqref{seceq}. The cost of storing $H_{k + 1}$ is lower than
that of~\eqref{gbroyd2}, since vectors $s_k$ and $y_k$ have already
been computed. The cost of calculating $H_{k + 1} v$ is higher: it
involves one scalar product, two sums of vector times a scalar and
$H_k v$. After $\ell$ updates of an initial approximation
$H_{k - \ell}$, current approximation $H_{k}$ is given by
\begin{equation}
  \label{bbroyd2}
  \begin{split}
    H_{k} & = H_{k - 1} V_{k - 1} + \frac{s_{k - 1} y_{k -
        1}^T}{\rho_{k - 1}} \\
    & = H_{k - \ell} \left( \prod_{j = k - \ell}^{k - 1}
      V_{j} \right) + \sum_{i = 1}^{\ell} \left( \frac{s_{ k - i}
        y_{ k - i}^T}{\rho_{ k - i}}
      \prod_{j = k - i + 1}^{k - 1} V_{j} \right)
  \end{split}
\end{equation}

Approach~\eqref{bbroyd1} has some advantages
over~\eqref{gbroyd2}. First, it does not need to compute $H_k v$ for
constructing the update. When $H_k$ is a complicated matrix, this is a
costly operation. Second, unlike~\eqref{gbroyd2}, matrices $V_{j}$
depend solely on $y_{j}$ and $s_{j}$ for all $j = 1, \dots, \ell$, so
it is possible to replace $H_{k - \ell}$ by different matrices without
updating the whole structure. This is suitable to be applied in a
limited-memory scheme~\cite{Gratton2016}. Third, the computation of
$H_{k} v$ can be efficiently implemented in a scheme similar to the
BFGS update described in~\cite{Nocedal2006}, as we show in
Algorithm~\ref{bbalg}. Unfortunately, the Broyden ``bad'' method is
known to behave worse in practice than the ``good''
method~\cite{DennisJr.1996}. To avoid the extra cost of computing
$H_{k} y_{k}$ in~\eqref{gbroyd2} it is common to compute a Cholesky or
LU factorization of $B_{k - \ell}$ and work directly
with~\eqref{gbroyd1}, performing rank-one updates of the
factorization, which can be efficiently implemented~\cite{Gill1974}.

\begin{algorithm}[h]
  \SetNlSty{textbf}{}{.}
  \DontPrintSemicolon

  \SetKwInOut{Data}{Data}
  \SetKwInOut{Input}{Input}
  \SetKwInOut{Output}{Output}

  \Data{ $H_{k - \ell} \in \R^{N \times N}$ and triples
    $(s_{k - j}, y_{k - j}, \rho_{k - j})$, for $j = 1, \dots, \ell$ }

  \Input{$v \in \R^N$}

  \Output{$r = H_{k} v$}

  \nl $q \leftarrow v$\;

  \lnl{bbalg:alpha} \For{$j = 1, \dots, \ell$}{

    \tcc*[l]{Store scalar $y_{k - j}^T (V_{k - j + 1} \cdots V_{k - 1}) v / \rho_{k - j}$}
    $\alpha_{j} \leftarrow (y_{k - j}^T q) / \rho_{k - j}$

    \tcc*[l]{Compute vector $(V_{k - j} \cdots V_{k - 1}) v$}
    $q \leftarrow q - \alpha_{j} y_{k - j}$

  }

  \lnl{bbalg:solving} $r \leftarrow H_{k - \ell} q$\;

  \nl \For{$i = 1, \dots, \ell$}{

    \tcc*[l]{Add the term
      $\left( y_{k - i}^T V_{k - i + 1} \cdots V_{k - 1} v / \rho_{k - i}
      \right) s_{k - i}$}
    $r \leftarrow r + \alpha_{i} s_{k - i}$ }

  \caption{Algorithm for matrix-vector multiplications on Broyden
    ``bad'' update.}
  \label{bbalg}
\end{algorithm} 

The class of rank-one least change secant updates can be generically
represented by updates of the form
\begin{equation}
  \label{secupdt}
  B_{k + 1} = B_k + \frac{(y_k - B_k s_k) w_k^T}{w_k^T s_k},
\end{equation}
where $w_k^T s_k \ne 0$. Setting $w_k = s_k$ defines the Broyden
``good'' method and $w_k = B_k^T y_k$ defines the Broyden ``bad''
method. Several other well known quasi-Newton methods fit in
update~\eqref{secupdt}, such as the Symmetric Rank-1 update used in
nonlinear optimization, which defines $w_k = y_k - B_k s_k$.
See~\cite{DennisJr.1979, DennisJr.1996} for details on least change
secant updates.

\section{A quasi-Newton approach for IP iterations}
\label{qn}

According to the general description of primal-dual IPMs in
Section~\ref{intro}, we can see that, at each iteration, they perform
one Newton step associated with the nonlinear system~\eqref{ipnls},
for decreasing values of $\mu$. Each step involves the computation of
the Jacobian of $F$ and the solution of a linear system~\eqref{ipmit}.

Our proposal for this work is to perform one quasi-Newton step to
solve~\eqref{ipnls}, replacing the true Jacobian $J(x, \lambda, z)$ by
a low rank approximation $B$. The idea might seem surprising at first
glance, since, for quadratic problems, $J(x, \lambda, z)$ is very
cheap to evaluate. In this section we further develop the quasi-Newton
ideas applied to interior point methods and show that they might help
to reduce the cost of the linear algebra when solving~\eqref{def:qp}.

It is important to note that $F$ and $J$ discussed in
Section~\ref{back} will be given by~\eqref{def:f} and~\eqref{def:j},
respectively, in the interior point context, which highlights the
importance of using the unreduced matrix in our analysis. Therefore,
variable $\qnv$ in Section~\ref{back} is given by $(x, \lambda, z)$
and, consequently, $N = 2n + m$.

\subsection{Initial approximation and update}

Suppose that $k \ge 0$ is an interior point iteration for which
system~\eqref{ipmit} was solved and $(\xysk{k + 1})$ was calculated,
using any available technique. Usually, solving~\eqref{ipmit} involves
an expensive factorization or the computation of a good preconditioner
associated with $J_k$.  Most traditional quasi-Newton methods
for general nonlinear systems compute $B_k$ by finite differences or
use a diagonal matrix as the initial approximation. According to
Section~\ref{back}, it is necessary to have an initial approximation
of $J_k$ in order to generate approximation $B_{k + 1}$ of $J_{k + 1}$
by low rank updates.  Most of traditional quasi-Newton methods for
general systems compute $B_k$ by finite differences or use a diagonal
matrix. Since $J_k$ have already been computed, we will define it as
$B_k$, i.e., the perfect approximation to $J_k$. It is clear that, in
such case, $H_k = J_k^{-1}$ is the approximation to $J_k^{-1}$.

In order to compute $B_{k + 1}$, vectors $s_k$ and $y_k$ in secant
equation~\eqref{seceq} have to be built:
\begin{equation}
  \label{skyk}
  \begin{aligned}
    s_{k} & = 
    \begin{bmatrix}
      s_{k,x} \\ s_{k,\lambda} \\ s_{k,z}
    \end{bmatrix}
    =
    \begin{bmatrix}
      x^{k + 1} - x^k \\ \lambda^{k + 1} - \lambda^k \\ z^{k + 1} - z^k
    \end{bmatrix}
    \\
    y_{k} & = 
    \begin{bmatrix}
      y_{k,c} \\ y_{k,b} \\ y_{k,\mu}
    \end{bmatrix}
    = F(\xysk{k + 1}) - F(\xysk{k}) \\
    & =
    \begin{bmatrix}
      - Q s_{k,x} + A^T s_{k,\lambda} + s_{k,z} \\ A s_{k,x} \\ X^{k + 1}
      Z^{k + 1} e - X^{k} Z^{k} e
    \end{bmatrix}
    .
  \end{aligned}
\end{equation}
The use of $J_k$ as the initial approximation ensures that the first
two block elements of $B_k s_k - y_k$ are zero. This is a well known
property of low rank updates given by~\eqref{secupdt} when applied to
linear functions (see~\cite[Ch. 8]{DennisJr.1996}). In
Lemma~\ref{l:struct} we show that rank-one secant updates maintain
most of the good sparsity structure of approximation $B_k$ when its
structure is similar to the true Jacobian of $F$.

\begin{lemma}
  \label{l:struct}
  Let $J$ be the Jacobian of $F$ given by~\eqref{def:f}. If the least
  change secant update $B_{k + 1}$ for approximating $J_{k + 1}$ is
  computed by~\eqref{secupdt} using $w_k^T =
  \begin{bmatrix}
    a_k & b_k & c_k
  \end{bmatrix}^T
  $, $a_k, c_k \in \R^n$, $b_k \in \R^m$, and $B_{k}$ is defined by
  \[
  B_{k} =
  \begin{bmatrix}
    - Q & A^T & I \\
      A & 0   & 0 \\
      M^1_{k} & M^2_{k} & M^3_{k}
  \end{bmatrix}
  \]
  then
  \[
  B_{k + 1} =
  \begin{bmatrix}
    - Q & A^T & I \\
      A & 0   & 0 \\
      M^1_{k + 1} & M^2_{k + 1} & M^3_{k + 1}
  \end{bmatrix},
  \]
  where $M^i_{k + 1}$ is a rank-one update of $M^i_{k}$, for
  $i = 1, 2, 3$. In addition, if $M^2_k = 0$ and $b_k = 0$, then
  $M^2_{k + 1} = 0$.
\end{lemma}

\begin{proof}
  By the definition of $s_k$ and $y_k$ in~\eqref{skyk} it is easy to
  see that $y_k - B_k s_k =
  \begin{bmatrix}
    0 & 0 & u_k
  \end{bmatrix}^T
  $, where
  \[
  u_k = (X^{k + 1} Z^{k + 1} - X^k Z^k) e - M^1_{k} s_{k, x} - M^2_{k}
  s_{k,\lambda} - M^3_{k} s_{k,z}.
  \]
  Using the secant update~\eqref{secupdt}, we have that the first two
  rows of $B_k$ are kept the same and
  \[
  \begin{aligned}
    M^1_{k + 1} & = M^1_k + u_k a_k^T / (w_k^T s_k) \\
    M^2_{k + 1} & = M^2_k + u_k b_k^T / (w_k^T s_k) \\
    M^3_{k + 1} & = M^3_k + u_k c_k^T / (w_k^T s_k).
  \end{aligned}
  \]
  It is easy to see that $M^2_{k + 1} = 0$ when $M^2_k = 0$ and
  $b_k = 0$.
\end{proof}

By Section~\ref{back} we know that Broyden ``good'' and ``bad''
updates are represented by specific choices of $w_k$ and, therefore,
enjoy the consequences of Lemma~\ref{l:struct}. Unfortunately, not
much can be said about the structure of the ``third row'' of
$B_{k + 1}$.  When $B_k = J_k$, the diagonal structure of blocks $Z^k$
and $X^k$, as well as the zero block in the middle, are likely to be
lost. However, if we select $ w_k^T =
\begin{bmatrix}
  s_{k,x} & 0 & s_{k,z}
\end{bmatrix}^T $,
then, by Lemma~\ref{l:struct}, the zero block is kept in $B_{k + 1}$.
The update given by this choice of $w_k$ is a particular case of
Schubert's quasi-Newton update for structured and sparse
problems~\cite{Schubert1970}. This update minimizes the distance to
$B_k$ on the space of the matrices that satisfy~\eqref{seceq} and have
the same block sparsity pattern of $B_k$~\cite{DennisJr.1979}. Using
the Sherman-Morrison-Woodbury formula, we also have the update for
$H_k$:
\[
H_{k + 1} = \left( I - \frac{(H_k y_k - s_k) w_k^T}{w_k^T H_k y_k}
\right) H_k^{-1},
\]
which only needs an extra computation of $H_k y_k$ to be stored. There
is no need to store $w_k$, since it is composed by components of
$s_k$. We can say that this approach is inspired in the Broyden
``good'' update.

On the other hand, if we use $w_k^T =
\begin{bmatrix}
  0 & y_{k, b} & y_{k, \mu}
\end{bmatrix}^T B_k $,
then we still have $M^2_{k + 1} = 0$ by Lemma~\ref{l:struct} and, in
addition, we are able to remove the calculation $H_k y_k$ in the
inverse. This approach is inspired by the Broyden ``bad'' update and
results in the following update
\begin{equation}
  \label{sbbroyd}
  H_{k + 1} = H_k + \frac{(s_k - H_k y_k) \begin{bmatrix}
      0 & y_{k, b} & y_{k, \mu}
    \end{bmatrix}^T
  }{y_{k, b}^T y_{k, b} + y_{k, \mu}^T y_{k, \mu}}.
\end{equation}
Up to the knowledge of the authors, this update has not been
theoretically studied in the literature.

Lemma~\ref{l:struct} also justifies our choice to work with
approximations of $J^{-1}$ rather than $J$. After $\ell > 0$ rank-one
updates, if $B_{k} u = v$ is solved by factorizations and backsolves,
it would be necessary to perform $\ell$ updates on the factorization
of initial matrix $B_{k - \ell}$, what could introduce many nonzero
elements. A clear benefit of defining
$B_{k - \ell} = J_{k - \ell}$ is that computing $H_{k} v$ uses the
already calculated factorizations/preconditioners for $B_{k - \ell}$,
which were originally used to solve~\eqref{ipmit} at iteration
$k - \ell$. Step~\ref{bbalg:solving} of Algorithm~\ref{bbalg} is an
example of low rank update~\eqref{bbroyd2}. Clearly, we do not
explicitly compute $H_{k - \ell} v$, but instead solve the system
$B_{k - \ell} u = v$.

\subsection{Computation of quasi-Newton steps}

Having defined how quasi-Newton updates are initialized and
constructed, we now have to insert the approximations in an interior
point framework. Denoting $(\xysk{0})$ as the starting point of the
algorithm, at the end of any iteration $k$ it is possible to build a
rank-one secant approximation of the unreduced matrix to be used at
iteration $k + 1$.  Let us consider iteration $k$, where $k \ge 0$ and
$\ell \ge 0$. If $\ell = 0$, then, by the previous subsection,
$B_{k - \ell} = B_k = J_k$ and the step in the interior point
iteration is the usual Newton step, given by~\eqref{ipmit}. If
$\ell > 0$, we have a quasi-Newton step, which can be viewed as a
generalization of~\eqref{ipmit}, and is computed by solving
\begin{equation}
  \label{qnipmit}
  B_k
  \begin{bmatrix}
    \dxk \\ \dyk \\ \dsk
  \end{bmatrix}
  =
  v
\end{equation}
or, equivalently, by performing $H_k v$. All the other steps of the
IPM remain exactly the same.

When $\ell > 0$, the cost of solving~\eqref{qnipmit} depends on the
type of update that is used. In general, it is the cost of solving
system $J_{k - \ell} r = q$ (or, equivalently, $J_{k - \ell}^{-1} q$)
plus some vector multiplications and additions. However, since
$J_{k - \ell}$ has already been the coefficient matrix of a linear
system at iteration $k - \ell$, it is usually less expensive than
solving for the first time. That is one of the main improvements that
a quasi-Newton approach brings to interior point methods.

When the Broyden ``bad'' update~\eqref{bbroyd2} is used together with
defining $B_{k - \ell} = J_{k - \ell}$ as the initial approximation,
it is possible to derive an alternative interpretation
of~\eqref{qnipmit}. Although this update is known to have worse
numerical behavior when compared with the ``good''
update~\eqref{gbroyd2}, this interpretation can result in a more
precise implementation, which is described in Lemma~\ref{l:bbroyd}.

\begin{lemma}
  \label{l:bbroyd}
  Assume that $k, \ell \ge 0$ and $H_{k}$ is the approximation of
  $J_{k}^{-1}$ constructed by $\ell$ updates~\eqref{bbroyd2} using
  initial approximation $H_{k - \ell} = J_{k - \ell}^{-1}$.  Given
  $v \in \R^{2n + m}$, the computation of $r = H_k v$ is equivalent to
  the solution of
  \begin{equation*}
    \label{l:bbroyd:e1}
    J_{k - \ell} r = v + 
    \begin{bmatrix}
      0 \\ 0 \\ \sum \limits_{i = 1}^{\ell} \alpha_i \left( Z^{k -
          \ell} s_{k - i, x} + X^{k - \ell} s_{k - i,z} - y_{k - i,
          \mu} \right)
    \end{bmatrix},
  \end{equation*}
  where
  $\displaystyle \alpha_i = \frac{y_{k - i}^T \prod_{j = k - i + 1}^{k
      - 1} V_{j}}{\rho_{k - i}} v$, for $i = 1, \dots, \ell$.
\end{lemma}

\begin{proof}
  Using the expansion~\eqref{bbroyd1} of Broyden ``bad'' update, the
  definition of $\alpha_i$ and the fact that $H_{k} = J_{k}^{-1}$, we
  have that
  \begin{equation}
    \label{l:bbroyd:e2}
    \begin{split}
      r & = H_k v = H_{k - \ell} \left( \prod_{j = k - \ell}^{k - 1}
        V_{j} \right) v + \sum_{i = 1}^{\ell} \left( \frac{s_{ k - i}
          y_{ k - i}^T}{\rho_{ k - i}}
        \prod_{j = k - i + 1}^{k - 1} V_{j} \right) v \\
      & = J_{k - \ell}^{-1} \left( \prod_{j = k - \ell}^{k - 1} V_{j}
      \right) v + \sum_{i = 1}^{\ell} \alpha_i s_{k -
        i} \\
      & = J_{k - \ell}^{-1} \left( v - \sum_{i = 1}^{\ell} \alpha_i
        y_{ k - i} \right) + \sum_{i = 1}^{\ell} \alpha_i s_{k - i},
    \end{split}
  \end{equation}
  where the last equality comes from the definition of $V_{k}$
  in~\eqref{bbroyd1}, applied recursively. When $i = 1$, we assume
  that $\prod_{j = k - i + 1}^{k - 1} V_{j}$ results in the identity
  matrix, therefore $\alpha_1 = y_{k - 1}^T v / \rho_{k - 1}$.
  Multiplying $J_{k - \ell}$ on the left on both sides
  of~\eqref{l:bbroyd:e2}, we obtain
  \[
  J_{k - \ell} r = v + \sum_{i = 1}^{\ell} \alpha_i \left( J_{ k -
      \ell} s_{k - i} - y_{k - i} \right).
  \]
  By Lemma~\ref{l:struct} and definition~\eqref{skyk}, the first two
  components of $J_{k - \ell} s_{k - i} - y_{k - i}$ are zero, for all
  $i$, which demonstrates the lemma.
\end{proof}

Lemma~\ref{l:bbroyd} states that only the third component of the right
hand side actually needs to be changed in order to compute Broyden
``bad'' quasi-Newton steps at iteration $k$. This structure is very
similar to corrector or multiple centrality correctors in IPMs and
reinforce the argument that the cost of computing a quasi-Newton step
is lower than the Newton step. It is important to note that scalars
$\alpha_i$ are the same as the ones computed at step~\ref{bbalg:alpha}
of Algorithm~\ref{bbalg}.

\subsection{Dealing with regularization}

Rank-deficiency of $A$, near singularity of $Q$ or the lack of strict
complementarity at the solution may cause matrix $J$, the augmented
system or the normal equations to become singular near the solution
of~\eqref{def:qp}. As the iterations advance, it becomes harder to
solve the linear systems. Regularization techniques address this issue
by adding small perturbations to $J$ in order to increase numerical
accuracy and convergence speed, without losing theoretical
properties. A common approach is to interpret the perturbation as the
addition of weighted proximal terms to the primal and dual
formulations of~\eqref{def:qp}. Saunders and
Tomlin~\cite{Saunders1996} consider fixed perturbations while Altman
and Gondzio~\cite{Altman1999} consider dynamic ones, computed at each
iteration. Friedlander and Orban~\cite{Friedlander2012} add extra
variables to the problem, expand the unreduced system and, after an
initial reduction, arrive in a regularized system similar
to~\cite{Altman1999}. In all these approaches, given reference points
$\hat x$ and $\hat \lambda$, the regularized matrix $J$
\begin{equation}
  \label{rj}
  J(x, \lambda, z) =
  \begin{bmatrix}
    - Q - R_p & A^T & I \\
    A &   R_d & 0 \\
    Z &   0 & X
  \end{bmatrix},
\end{equation}
where diagonal matrices $R_p \in \R^{n \times n}$ and
$R_d \in \R^{m \times m}$ represent primal and dual regularization,
respectively, can be viewed as the Jacobian of the following function
\[
\hat F(x, \lambda, z) = 
\begin{bmatrix}
  A^T \lambda - Q x - R_p (x - \hat x) - c \\
  A x + R_d (\lambda - \hat \lambda) - b \\
  X Z e
\end{bmatrix}.
\]
Any choice is possible for reference points $\hat x$ and
$\hat \lambda$. However, in order to solve the original Newton
system~\eqref{ipmit} and make use of the good properties of the
regularization~\eqref{rj} at the same time, they are usually set to
the current iteration points $x^k$ and $\lambda^k$, respectively,
which annihilates terms $R_p (x - \hat x)$ and
$R_d (\lambda - \hat \lambda)$ on the right hand side of~\eqref{ipmit}
during affine scaling steps.

Matrix $J$ given by~\eqref{rj} now depends on $R_p$ and $R_d$ in
addition to $x$ and $z$. The regularization terms $R_p$ and $R_d$ do
not need to be considered as variables, but if new regularization
parameters are used, a new factorization or preconditioner needs to be
computed. Since this is one of the most expensive tasks of the IP
iteration, during quasi-Newton step $k$ the regularization parameters
are not allowed to change from those selected at iteration $k - \ell$,
where the initial approximation was selected. That is a reasonable
decision, as the system that is actually being solved in practice has
the coefficient matrix from iteration $k - \ell$. The fact that the
regularization terms are linear in $\hat F$ implies, by
Lemma~\ref{l:struct}, that the structure of~\eqref{rj} is maintained
during least change secant updates.

The reference points have no influence in $J$, but they do influence
the function $\hat F$. Suppose, as an example, that $\ell = k$, i.e.,
the initial approximation for quasi-Newton is the Jacobian at the
starting point $(\xysk{0})$, and only quasi-Newton steps are taken in
the interior point algorithm. If we use $x^0$ and $\lambda^0$ as the
reference points and the algorithm converges, the limit point could be
very different from the true solution, as initial points usually are
far away from the solution, especially for infeasible IPMs. If we
update the reference points at each quasi-Newton iteration, as it is
usually the choice in literature~\cite{Altman1999, Friedlander2012},
we eliminate their effect on the right hand side of~\eqref{qnipmit}
during affine scaling steps. By~\eqref{linmodel}, $B_{k + 1}$
is the Jacobian of a linear approximation of $\hat F$ which
interpolates $(\xysk{k})$ and $(\xysk{k + 1})$. As the regularization
parameters are fixed during quasi-Newton iterations, the reference
points can be seen as simple constant shifts on $\hat F$, with no
effect on the Jacobian. Therefore, the only request is that $\hat F$
has to be evaluated at points $(\xysk{k})$ and $(\xysk{k + 1})$ using
the same reference points, when calculating $y_k$ by~\eqref{skyk}. The
effect of changing the reference points at each iteration in practice
is the extra evaluation of $\hat F$ at the beginning of iteration $k$.

\section{Implementation}
\label{implementation}

The quasi-Newton approach can easily be inserted into an existing
interior point method implementation. In this work, the primal-dual
interior point algorithm \hopdm~\cite{Gondzio1995} was modified to
implement the quasi-Newton approach.  Algorithm~\ref{qnipm} describes
the steps of a conceptual quasi-Newton primal-dual interior point
algorithm.

\begin{algorithm}[h]
  \SetNlSty{textbf}{}{.}
  \DontPrintSemicolon

  \SetKwInOut{Data}{Initialization}
  \SetKwInOut{Input}{Input}
  \SetKwInOut{Output}{Output}

  \Data{ $F$, $J$ and $(\xysk{0})$. Set $k \leftarrow 0$ and
    $\ell \leftarrow 0$. }

  \BlankLine

  \lnl{qnipm:sys} Solve system~\eqref{qnipmit} with different right
  hand sizes, if necessary, to compute step $(\dxk, \dyk, \dsk)$\;

  \BlankLine

  \nl Calculate $\alpha_P^k$ and $\alpha_D^k$ such that
  $(\xysk{k + 1})$ given by~\eqref{ipmss} satisfy
  $x^{k + 1}, \lambda^{k + 1} > 0$\;

  \BlankLine

  \lnl{qnipm:qn} Compute $s_k$ and $y_k$ by~\eqref{skyk}\;
  \BlankLine
  \If{will store quasi-Newton information,}{
    Store appropriate quasi-Newton information\;
    $\ell \leftarrow \ell + 1$}
  \Else{$\ell \leftarrow 0$}

  \nl $k \leftarrow k + 1$ and go back to step~\ref{qnipm:sys}\;

  \caption{Quasi-Newton Interior Point algorithm}
  \label{qnipm}
\end{algorithm} 

The most important element of Algorithm~\ref{qnipm} is $\ell$, the
memory size of the low rank update, which controls if the iteration
involves Newton or quasi-Newton steps. At step~\ref{qnipm:sys} several
systems~\eqref{qnipmit} might be solved, depending on the IPM
used. \hopdm\ implements the strategy of multiple centrality
correctors~\cite{Colombo2008}, which tries to maximize the step-size
at the iteration. \hopdm\ also implements the regularization
strategy~\eqref{rj}. Note in~\eqref{qnipmit} that we do not have to
care how the systems are solved, only how to implement the
matrix-vector multiplication $H_k v$ efficiently.

Step~\ref{qnipm:qn} is the most important step in a quasi-Newton IP
algorithm, since it decides whether or not quasi-Newton steps will be
used in the next iteration. Several possible strategies are discussed
in this section, as well as some implementation details.

Bound constraints
\begin{equation*}
  l \le x \le u, \quad l, u \in \R^n
\end{equation*}
can be considered in the general definition~\eqref{def:qp} of a
quadratic programming problem by using slack variables. \hopdm\
explicitly deals with bound constraints and increases the number of
variables to $4n + m$. When bound constraints are considered, function
$F$ is given by
\[
F(x, t, \lambda, z, w) =
\begin{bmatrix}
  A^T \lambda - Q x + z - w - c \\
  A x - b \\
  x + t - u \\
  X Z e \\
  T W e
\end{bmatrix}
\]
and the Jacobian $J$ is
\[
J(x, t, \lambda, z, w) = 
\begin{bmatrix}
  -Q & 0 & A^T & I & -I \\
  A  & 0 & 0   & 0 & 0 \\
  I  & I & 0   & 0 & 0 \\
  Z  & 0 & 0   & X & 0 \\
  0  & W & 0   & 0 & T
\end{bmatrix}.
\]
Note that, in this case, $l$ is eliminated by proper shifts, $u$
represents upper shifted constraints and $t$ represents slacks. All
the results and discussions considered so far can be easily adapted to
the bound-constrained case. Therefore, in order to keep notation
simple, we will refer to the more general and simpler
formulation~\eqref{def:qp} and work in the $(2n + m)$-dimensional
space.

\subsection{Storage of $H_{k}$ and computation of
  $H_{k} v$}

When solving quadratic problems, the Jacobian of function $F$ used in
a primal-dual interior point method is not expensive to compute and
has an excellent structure, which can be efficiently explored by
traditional approaches. Therefore, there is no point in explicitly
building approximation matrix $B_{k}$ (or $H_{k}$) since, by
Lemma~\ref{l:struct}, they would be denser. For an efficient
implementation of the algorithm only the computation $H_{k} v$ has to
be performed in~\eqref{qnipmit}. To accomplish this task, we store
\begin{itemize}
\item Initial approximation $J_{k - \ell}$ and
\item Triples $(s_{k - i}, u_{k - i}, \rho_{k - i})$ or
  $(s_{k - i}, y_{k - i}, \rho_{k - i})$ , $i = 1, \dots, \ell$, if
  updates are based on Broyden ``good'' or ``bad'' method,
  respectively.
\end{itemize}

In order to store $J_{k - \ell}$ we have to store vectors
$x^{k - \ell}$ and $\lambda^{k - \ell}$, since all other blocks of $J$
are constant. If regularization is being used, vectors $R_p$ and $R_d$
used at iteration $k - \ell$ are also stored. The reference points are
not stored. The most important structure to store is the factorization
or the preconditioner computed when solving~\eqref{qnipmit} at
iteration $k - \ell$ for the first time. Without this information, the
computation of $H_k v$ would have the same computational cost of using
the true matrix $J_{k}$. Data is stored at step~\ref{qnipm:qn} of
Algorithm~\ref{qnipm}, whenever it has decided to store quasi-Newton
information and $\ell = 0$.

Regarding the triples, they are composed of two $(2n + m)$-dimensional
vectors and one scalar. Storing $y_{k - i}$ is the most expensive part
in Broyden ``bad'' updates, since function $F$ has to be evaluated
twice. In Broyden ``good'' updates the computation of $u_{k - i}$ is
the most expensive, due to the computation of $H_{k - i} y_{k - i}$.

The implementation of an algorithm to compute $H_k v$ depends on the
selected type of low rank update. Algorithm~\ref{bbalg} is an
efficient implementation of the general Broyden ``bad''
update~\eqref{bbroyd2}. If the structure described by
Lemma~\ref{l:struct} is being used, then all vector multiplications
are performed before the solution of the linear system, as described
by Algorithm~\ref{sbbalg}. Both algorithms can be easily modified to
use updates of the form
$w_k^T = \begin{bmatrix} a_k & b_k & c_k
\end{bmatrix}^T
B_k$
in the generic update~\eqref{secupdt}. The only changes are the
storage of an extra vector and the computation of scalars $\alpha_i$
at step~\ref{bbalg:alpha}. The implementation of the sparse
update~\eqref{sbbroyd} is straightforward and there is no need to
store extra information. Algorithm~\ref{sbbalg} uses a little extra
computation, since vector $q$ is discarded after the computation of
all $\alpha_i$. On the other hand, there is no need to store blocks
$s_{k - i,\lambda}$, $i = 1, \dots, \ell$.

\begin{algorithm}[h]
  \SetNlSty{textbf}{}{.}
  \DontPrintSemicolon

  \SetKwInOut{Data}{Data}
  \SetKwInOut{Input}{Input}
  \SetKwInOut{Output}{Output}

  \Data{ $J_{k - \ell} = J(\xysk{k - \ell})$ and $(s_{k - i}, y_{k - i}, \rho_{k - i})$,
    for $i = 1, \dots, \ell$ }

  \Input{$v \in \R^{2n + m}$}

  \Output{$r = H_{k} v$}

  \nl $q \leftarrow v$\;

  \nl \For{$i = 1, \dots, \ell$}{

    $\alpha_{i} \leftarrow (y_{k - i}^T q) / \rho_{k - i}$

    $q \leftarrow q - \alpha_{i} y_{k - i}$

  }

  \nl $q \leftarrow v$\;

  \nl \For{$i = 0, \dots, \ell - 1$}{

    $q \leftarrow q + 
    \begin{bmatrix}
      0 \\ 0 \\ \alpha_{i} \left( Z^{k - \ell} s_{k - i, x} + X^{k - \ell} s_{k
          - i, z} - y_{k - i, \mu} \right)
    \end{bmatrix}$ }

  \nl Solve $J_{k - \ell} r = q$\;

  \caption{Algorithm for matrix-vector multiplications in Broyden
    ``bad'' update using structural information}
  \label{sbbalg}
\end{algorithm} 

Algorithm~\ref{gbalg} describes the steps to compute $H_{k} v$ when
Broyden ``good'' update~\eqref{gbroyd2} is considered. Note that a
linear system is first solved, then a sequence of vector
multiplications and additions is applied. The algorithm is simpler and
more general than Algorithm~\ref{bbalg}, but it has to be called more
often in an interior point algorithm: to compute the steps
(step~\ref{qnipm:sys} in Algorithm~\ref{qnipm}) and to compute
$H_{k} y_{k}$, needed to build $u_{k}$ (step~\ref{qnipm:qn} in
Algorithm~\ref{qnipm}).  Algorithm~\ref{gbalg} is very general and can
be easily modified to consider any least change secant update of the
form~\eqref{secupdt} without extra storage requirements, although not
necessarily in an efficient way.

\begin{algorithm}[h]
  \SetNlSty{textbf}{}{.}
  \DontPrintSemicolon

  \SetKwInOut{Data}{Data}
  \SetKwInOut{Input}{Input}
  \SetKwInOut{Output}{Output}

  \Data{ $J_{k - \ell} = J(\xysk{k - \ell})$ and
    $(s_{k - i}, u_{k - i}, \rho_{k - i})$, for
    $i = 1, \dots, \ell$, as described in~\eqref{gbroyd2} }

  \Input{$v \in \R^{2n + m}$}

  \Output{$r = H_{k} v$}

  \nl Solve $J_{k - \ell} q = v$\;

  \nl $r \leftarrow q$\;

  \nl \For{$i = 1, \dots, \ell$}{

    $\alpha_{i} \leftarrow (s_{k - i}^T r) / \rho_{k - i}$

    $r \leftarrow r + \alpha_{i} u_{k - i}$

  }

  \caption{Algorithm for matrix-vector multiplications in Broyden
    ``good'' update}
  \label{gbalg}
\end{algorithm} 

\subsection{Size of $\ell$}

The cost of computing $H_k v$ increases as the quasi-Newton memory
$\ell$ increases. In addition, it was observed that the quality of the
approximation decreases when the quasi-Newton memory is
large~\cite{Bergamaschi2018}. In our implementation of
Algorithm~\ref{qnipm}, we also observed the decrease in the quality of
the steps when $\ell$ is too large. The decrease of the barrier
parameter $\mu_k = {x^k}^T z^k / n$ for different bounds on $\ell$ is
shown in Figure~\ref{fig:l}, for problem \texttt{afiro}, the smallest
example in Netlib test collection. In this example, Newton steps were
allowed after $\ell_\mathrm{max}$ quasi-Newton iterations, where
$\ell_\mathrm{max} \in \{0, 5, 20, 100, 200\}$. The maximum of 200
iterations was allowed.

We can see that if the Jacobian is only evaluated once
($\ell_\mathrm{max} = 200$) then the method is unable to converge in
200 iterations. As the maximum memory is reduced, the number of
iterations to convergence is also reduced. On the other hand, the
number of (possibly expensive) Newton steps is increased. When
$\ell_\mathrm{max} = 0$, i.e., no quasi-Newton steps, the algorithm
converges in 7 iterations. We take the same approach
as~\cite{Bergamaschi2018} and define an upper bound
$\ell_\mathrm{max}$ on $\ell$ in the implementation of
Algorithm~\ref{qnipm}. When this upper bound is reached, we set $\ell$
to 0, which, by~\eqref{qnipmit}, results in the computation of a
Newton step. The verification is performed at step~\ref{qnipm:qn} of
Algorithm~\ref{qnipm}.
This approach is also known as quasi-Newton with
restarts~\cite{Luksan1998} and differs from usual limited-memory
quasi-Newton~\cite{Nocedal2006}, where only the oldest information is
dropped.

\begin{figure}[ht]
  \centering
  \includegraphics{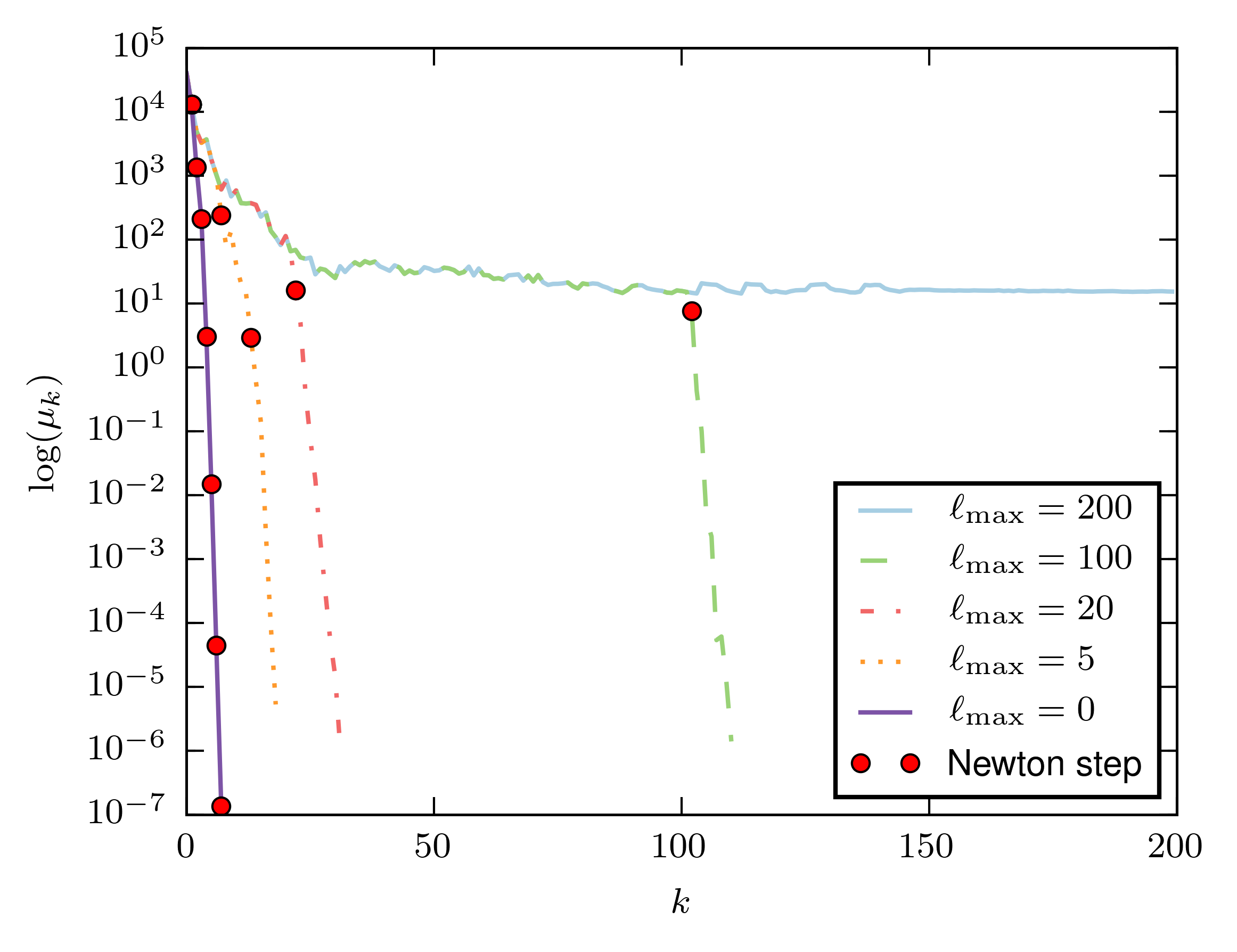}
  \caption{Small bounds for $\ell$ reduce the number of iterations,
    but increase the necessity of evaluating and factorizing the
    Jacobian. The circles represent iterations where Newton steps were
    calculated.}
  \label{fig:l}
\end{figure}

\subsection{The quasi-Newton steps}

The behavior of consecutive quasi-Newton steps depicted in
Figure~\ref{fig:l} reminds us that it is important to use the true
Jacobian in order to improve convergence of the method. However, we
would like to minimize the number of times the Jacobian is evaluated,
since it involves expensive factorizations and
computations. Unfortunately, to use only the memory bound as a
criterion to compute quasi-Newton steps is not a reasonable
choice. When $\ell_\mathrm{max} = 100$, for example, the algorithm
converges in 110 iterations, but it spends around 60 iterations
without any improvement. As the dimension of the problem increases,
this behavior is getting even worse. We can also see that the choice
$\ell_\mathrm{max} = 20$ is better for this problem, as the algorithm
converges in 31 iterations, computing only two times the Cholesky
factorization of the Jacobian.

The lack of reduction is related to small step-sizes $\alpha_P^k$ and
$\alpha_D^k$. Our numerical experience with quasi-Newton IP methods
indicates that the quasi-Newton steps often are strongly attracted to
the boundaries. The step-sizes calculated for directions originated
from a quasi-Newton predictor-corrector strategy are almost always
small and need to be fixed. Several strategies have been tried to
increase the step-sizes of those steps:
\begin{enumerate}[(i)]
\item \label{str:double} Perturb complementarity pairs $x_i z_i$
  for which the relative component-wise direction magnitude
  \begin{equation}
    \label{cwqnerror}
    \frac{|\left[ \dxk \right]_i|}{x_i^k}\quad \text{or}\quad
    \frac{|\left[ \dsk \right]_i|}{z_i^k},\quad i = 1, \dots, n
  \end{equation}
  is high and then recompute quasi-Newton direction;
\item \label{str:mc} Use multiple centrality
  correctors~\cite{Colombo2008};
\item \label{str:gent} Gentle reduction of $\mu$ on quasi-Newton
  iterations, selecting $\sigma_k$ close to 1 in the predictor and
  corrector steps.
\end{enumerate}
Note that the terms in~\eqref{str:double} are the inverse of the
maximum step-size allowed by each component.

The motivation of strategy~\eqref{str:double} is the strong relation
observed between components of the quasi-Newton direction which are
too large with respect their associated variable and components which
differ too much from the respective component of the Newton direction
for the same iteration, i.e.,
\begin{equation}
  \label{cwnqnerror}
  \frac{\left| \left[ {\dxk}^{(N)} - {\dxk}^{(QN)} \right]_i
    \right|}{x^k_i}\quad \text{and}\quad \frac{\left| \left[
        {\dsk}^{(N)} - {\dsk}^{(QN)} \right]_i
    \right|}{z^k_i},\quad i = 1, \dots, n.
\end{equation}
We display this relation in Figure~\ref{fig:alpha}(a) for one
iteration on linear problem \texttt{GE}. Positive spikes represent the
component-wise relative magnitude of quasi-Newton
steps~\eqref{cwqnerror} for each component of variables $x$ and
$z$. The higher the spikes, the smaller the step-sizes are. Negative
spikes represent the component-wise relative error between the Newton
and quasi-Newton directions~\eqref{cwnqnerror}. The lower the spikes,
the larger the relative difference between Newton and quasi-Newton
components. To generate this figure, the problem was solved twice and,
at the selected iteration, the Newton step and quasi-Newton step were
saved. Only negative quasi-Newton directions were considered in the
figure. It is possible to see in Figure~\ref{fig:alpha}(a) that very
few components are responsible for the small
step-sizes. Interestingly, most of those blocking components are
associated with components of the quasi-Newton direction which differ
considerably from the Newton direction.  Unfortunately, numerical
experiments show that the perturbation of variables or setting the
problematic components to zero has the drawback of increasing the
infeasibility and cannot be performed at every iteration.

\begin{figure}[ht]
  \centering
  
  \begin{tabular}{cc}
    \includegraphics{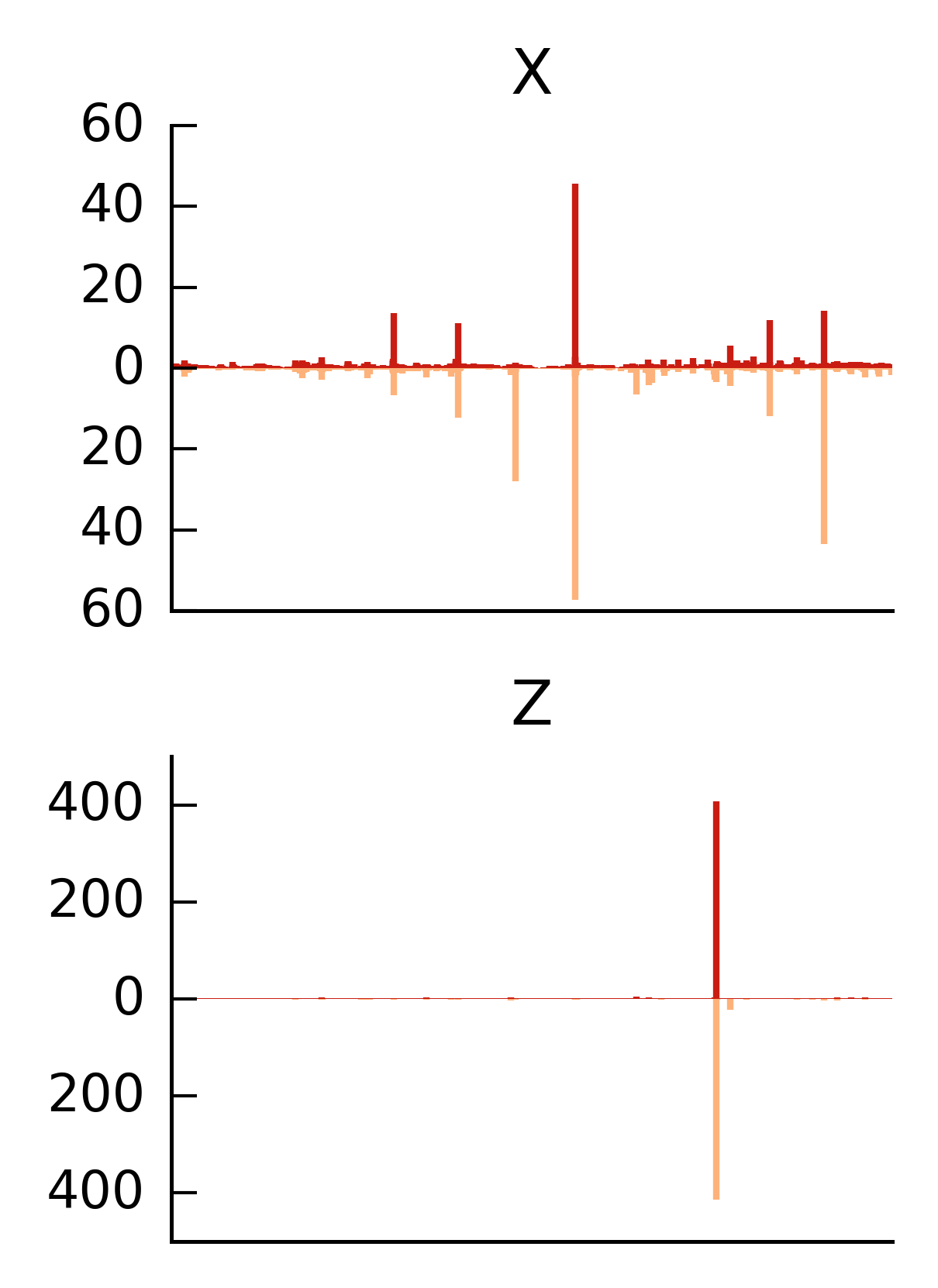} &
    \includegraphics{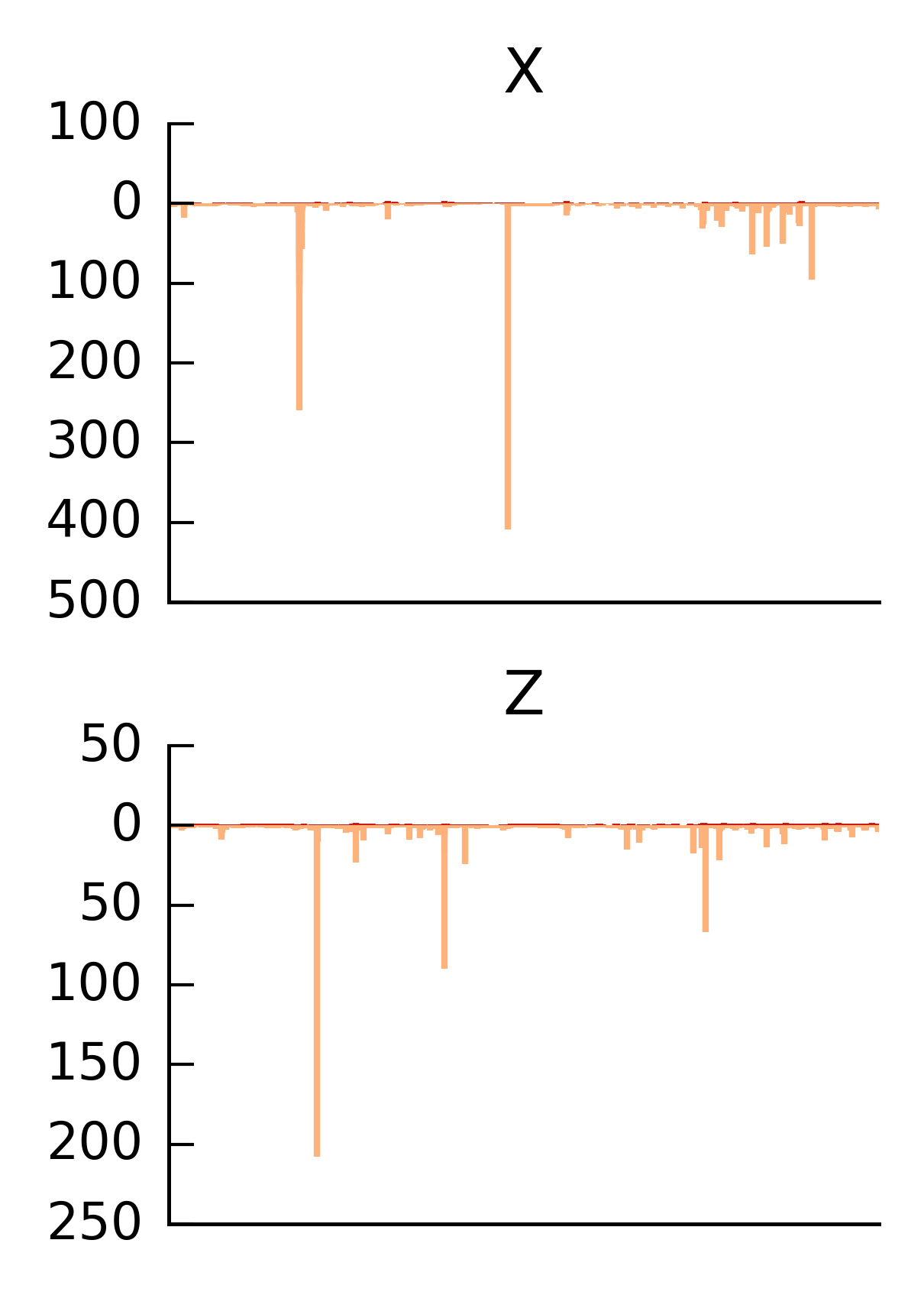} \\
    (a) & (b) \\
  \end{tabular}

  \caption{Relation between small step-sizes for quasi-Newton steps
    (positive spikes) and large relative errors when compared with
    Newton step (negative spikes) for one iteration on linear problem
    \texttt{GE}. High positive spikes represent blocking components of
    the quasi-Newton direction. The errors when only a simple
    predictor-corrector direction is used are displayed in (a). The
    effect of using strategy~\eqref{str:mc} to improve step-sizes is
    shown in (b).}
  \label{fig:alpha}
\end{figure}

To test the impact of each strategy on the quality of the steps, four
linear programming problems were selected: \texttt{afiro},
\texttt{GE}, \texttt{stocfor3} and \texttt{finnis}. The tests were
performed as follows. Given an iteration $k$ of a problem, we run
algorithm \hopdm\ allowing only Newton steps up to iteration $k - 1$.
At iteration $k$ only one of each approach is applied: Newton step,
quasi-Newton step, or one of the discussed strategies
\eqref{str:double}, \eqref{str:mc} or \eqref{str:gent}. Only one
affine-scaling predictor and one corrector were allowed, except for
strategy~\eqref{str:mc}, where multiple centrality correctors were
used at iteration $k$. We repeated this procedure for $k$ from 2 up to
the total number of iterations that the original version of \hopdm\
needed to declare convergence.

The average of the sum of the step-sizes for each problem and for each
approach is shown in Table~\ref{tab:alpha}. We can see that
quasi-Newton steps are considerably smaller than Newton steps. All
improvement strategies are able to increase, on average, the sum of
the step-sizes. Strategy~\eqref{str:double} has the drawback of
increasing the infeasibility and has a huge impact on the convergence
of the algorithm. Strategy~\eqref{str:gent} is simple and efficient to
implement but has worse results when compared to
strategy~\eqref{str:mc}, based on multiple centrality
correctors. Strategy~\eqref{str:mc} has the ability to improve
quasi-Newton directions in almost all iterations and has the drawback
of extra backsolves. Similar behavior was observed
in~\cite{Colombo2008}. The effect of strategy~\eqref{str:mc} is shown
in Figure~\ref{fig:alpha}(b). Step-sizes are increased, but the new
quasi-Newton direction is slightly different from the Newton direction
for the same step. Strategy~\eqref{str:mc} was selected as the default
one in our implementation.

\begin{table}[ht]
  \centering
  \begin{tabular}{lrrrrr}
    \toprule
     & Newton & Quasi-Newton & \eqref{str:double} & \eqref{str:mc}  & \eqref{str:gent} \\
    \midrule
    afiro    &  1.826500 &  0.849070 &  1.065883 &  $\mathbf{1.280400}$ &  0.908283 \\
    GE       &  0.911343 &  0.079197 &  0.264640 &  $\mathbf{0.620266}$ &  0.142124 \\
    stocfor3 &  1.006294 &  0.089973 &  0.569176 &  $\mathbf{1.163839}$ &  0.386584 \\
    finnis   &  1.454824 &  0.074488 &  0.452727 &  $\mathbf{1.059195}$ &  0.405455 \\
    \bottomrule
\end{tabular}

\caption{Average of the sum $\alphapk + \alphadk$ for different improvement
  strategies on selected linear programming problems. The use of multiple
  centrality correctors (strategy~\eqref{str:mc}) resulted in values similar
  to the Newton step.}
  \label{tab:alpha}
\end{table}

In order to perform as few Newton steps as possible,
step~\ref{qnipm:qn} of Algorithm~\ref{qnipm} has to be carefully
implemented. Clearly, the first basic condition to try a quasi-Newton
step at iteration $k + 1$, $k \ge 0$, is to check if there is
available memory to store it at iteration $k$.

\begin{criterion}[Memory criterion]
  \label{crit:mem}
  If $\ell \le \ell_\mathrm{max}$.
\end{criterion}

Our experience shows that quasi-Newton steps should always be tried,
since they are cheaper than Newton steps. This means that a
quasi-Newton step is always tried (but not necessarily accepted) after
a Newton step in the present implementation. As shown in
Figure~\ref{fig:l}, using only Criterion~\ref{crit:mem} can lead to
slow convergence and slow convergence is closely related to small
step-sizes. Therefore, in addition to Criterion~\ref{crit:mem} we
tested two criteria, which cannot be used together. In
Section~\ref{numerical} we compare those different acceptance
criteria.

\begin{criterion}[$\alpha$ criterion]
  \label{crit:alpha}
  If iteration $k$ is a quasi-Newton iteration and
  \[
  \alphapk + \alphadk \ge \varepsilon_\alpha.
  \]
\end{criterion}

\begin{criterion}[Centrality criterion]
  \label{crit:cent}
  If iteration $k$ is a quasi-Newton iteration and
  \[
  {x^{k + 1}}^T z^{k + 1} \le \varepsilon_c \left( {x^k}^T z^k
  \right).
  \]
\end{criterion}

\section{Numerical results}
\label{numerical}

Algorithm~\ref{qnipm} was implemented in Fortran 77 as a modification
of the primal-dual interior point algorithm \hopdm~\cite{Gondzio1995},
release 2.45. The code was compiled using \texttt{gfortran} 4.8.5 and
run in a Dell PowerEdge R830 powered with Red Hat Enterprise Linux, 4
processors Intel Xeon E7-4660 v4 2.2GHz and 512GB RAM.  The
modifications discussed in Sections~\ref{qn} and~\ref{implementation}
have been performed in order to accommodate the quasi-Newton
strategy. The main stopping criteria have been set to Mehrotra and
Li's stopping criteria~\cite{Colombo2008, Mehrotra2005}:
\begin{equation}
  \label{converged}
  \frac{\mu}{1 + |c^T x|} \le \varepsilon_\mathrm{opt},
  \quad \frac{\|b - Ax\|}{1 + \|b\|} \le \varepsilon_P, \quad
  \frac{\|c - A^T \lambda - z\|}{1 + \|c\|} \le \varepsilon_D,
\end{equation}
where $\mu = x^T z / n$. By default, in \hopdm\ parameters are defined
to $\varepsilon_\mathrm{opt} = 10^{-10}$, $\varepsilon_P = 10^{-8}$
and $\varepsilon_D$ is set to $10^{-8}$ for linear problems and to
$10^{-6}$ for quadratic problems. In addition to~\eqref{converged},
successful convergence is also declared when lack of improvement is
detected and $\mu / (1 + |c^T x|) \le 10^3 \varepsilon_\mathrm{opt}$.
Besides several performance heuristics, \hopdm\ implements the
regularization technique~\cite{Altman1999} and the multiple centrality
correctors strategy~\cite{Colombo2008}. When solving systems with the
unreduced matrix, sparse Cholesky factorization of normal equations or
$LDL^T$ factorization of the augmented system is automatically
selected on initialization. \hopdm\ also has a
matrix-free~\cite{Gondzio2012c} implementation for which the present
approach is fully compatible.


According to Algorithm~\ref{qnipm}, once a quasi-Newton step is
computed, it is used to build point $(\xysk{k + 1})$. However, in
practice, if such step is considered ``bad'', it is also possible to
discard it, setting $\ell = 0$, compute the exact Jacobian and perform
the Newton step at this iteration. The idea is to avoid quasi-Newton
steps which might degrade the quality of the current
point. Preliminary experiments using linear programming problems from
Netlib collection were performed, in order to test several
possibilities for $\ell_\mathrm{max}$ in Criterion~\ref{crit:mem} and
to select between Criteria~\ref{crit:alpha} and~\ref{crit:cent}. In
addition we also verified the possibility to reject quasi-Newton
steps, instead of always accepting them. The selected combination uses
$\ell_\mathrm{max} = 5$ and Criterion~\ref{crit:cent} with
$\varepsilon_c = 0.99$. Rejecting quasi-Newton steps has not led to
reductions in the number of factorizations and has the drawback of
more expensive iterations, therefore, the steps are always taken. As
mentioned in Section~\ref{implementation}, the multiple centrality
correctors strategy~\eqref{str:mc} is used to improve quasi-Newton
directions.

A key comparison concerns the type of low rank update to be
used. Three implementations were tested:
\begin{itemize}
\item[\texttt{U1}] General Broyden ``bad'' algorithm, described by
  Algorithm~\ref{bbalg};
\item[\texttt{U2}] Sparse Broyden ``bad'' algorithm, described by
  Algorithm~\ref{sbbalg} using update~\eqref{sbbroyd} inspired in
  Schubert's update~\cite{Schubert1970};
\item[\texttt{U3}] General Broyden ``good'' algorithm, described by
  Algorithm~\ref{gbalg}.
\end{itemize}

Four test sets were used in the comparison: 96 linear problems from
Netlib\footnote{\url{http://www.netlib.org/lp/data/}}, 10 medium-sized
linear problems from Maros-Mészáros \texttt{misc}
library\footnote{\url{http://old.sztaki.hu/~meszaros/public_ftp/lptestset/misc/}},
39 linear problems from the linear relaxation of Quadratic Assignment
Problems
(QAP)\footnote{\url{http://anjos.mgi.polymtl.ca/qaplib/inst.html}} and
138 convex quadratic programming problems from Maros-Mészáros
\texttt{qpdata}
library\footnote{\url{http://old.sztaki.hu/~meszaros/public_ftp/qpdata/}}.
In order to compare algorithms in large test sets, performance
profiles were used~\cite{Dolan2002}. A problem is declared
solved by an algorithm if the obtained solution $(\xysk{*})$
satisfies~\eqref{converged}.  Number of factorizations or total CPU
time are used as performance measures.

Using the default \hopdm\ values for~\eqref{converged},
implementations \texttt{U1}, \texttt{U2} and \texttt{U3} are able to
solve 269, 275 and 271 problems, respectively, out of 283. There were
19 problems in which at least one implementation did not solve. We
relaxed the parameters in~\eqref{converged}, multiplying them by a
factor of $10^2$, and solved the 19 problems again. The resulting
performance profiles in numbers are shown in Table~\ref{tab:updt},
using number of factorizations and CPU time as performance
measures. The efficiency of an algorithm is the number of solved
problems in which the algorithm spent the smallest number of
factorizations (or the smallest amount of CPU time) among the compared
algorithms. The robustness is the total number of problems solved.

We can see that update \texttt{U2} solves 210 problems using the
smallest number of factorizations and 137 problems using least CPU
time, while \texttt{U1} solves 177 and 126 and \texttt{U3} solves 123
and 85, respectively. In addition, updates \texttt{U2} and \texttt{U3}
are the most robust implementations, being able to solve 281 out of
283 problems. Therefore, \texttt{U2} was used as the default update in
this work. Update \texttt{U2} has performed particularly well on
quadratic problems, what explains the difference in efficiency between
updates.

\begin{table}[ht!]
  \centering

\begin{tabular}{lrrr}
    \toprule
    {} &  Efficiency   &  Efficiency & Robustness \\
    {} & Factorization &  CPU time   &            \\
    \midrule
    \texttt{U1} &     177 &     126 &        280 \\
    \texttt{U2} &     210 &     137 &        281 \\
    \texttt{U3} &     123 &      85 &        281 \\
    \bottomrule
  \end{tabular}
  \caption{Performance profiles for implementations \texttt{U1}, \texttt{U2} and
    \texttt{U3} on all the 283 small- and medium-sized test problems
    considered in this work.}
  \label{tab:updt}
\end{table}


Based on the preliminary results, the default implementation of
Algorithm~\ref{qnipm}, denoted \qnipm\ from now on, uses update
\texttt{U2} for solving~\eqref{qnipmit} and computing the step,
strategy~\eqref{str:mc} to improve quasi-Newton directions and
Criteria~\ref{crit:mem} and~\ref{crit:cent} to decide when to use
quasi-Newton at step~\ref{qnipm:qn}. By default, \hopdm\ uses multiple
centrality correctors, which were shown to improve convergence of the
algorithm~\cite{Colombo2008}. We implemented two versions of
Algorithm~\ref{qnipm}: with (\qnipmmc) and without (\qnipm) multiple
centrality correctors for computing Newton steps. Since we are using
strategy~\eqref{str:mc}, multiple correctors are always used for
quasi-Newton steps. Each implementation was compared against its
respective original version: \hopdm-mc and \hopdm.

In the first round of tests only the QAP collection was excluded from
the comparison, which gives 244 problems from Netlib and from
Maros-Mészáros linear and quadratic programming test collection. The
performance profiles using number of factorizations and CPU time as
performance measures are shown in
Figure~\ref{fig:pphopdm}. Comparisons between the implementation of
\hopdm\ without multiple centrality correctors and \qnipm\ are given
by Figures~\ref{fig:pphopdm}(a) and~\ref{fig:pphopdm}(b). The
comparison of implementations \hopdm-mc and \qnipmmc\ is displayed in
Figures~\ref{fig:pphopdm}(c) and~\ref{fig:pphopdm}(d).

\begin{figure}[ht!]
  \centering
  
  \begin{tabular}{cc}
    \midrule
    \multicolumn{2}{c}{Without multiple centrality correctors}\\ \midrule
    Cholesky factorizations & CPU time \\
    \includegraphics{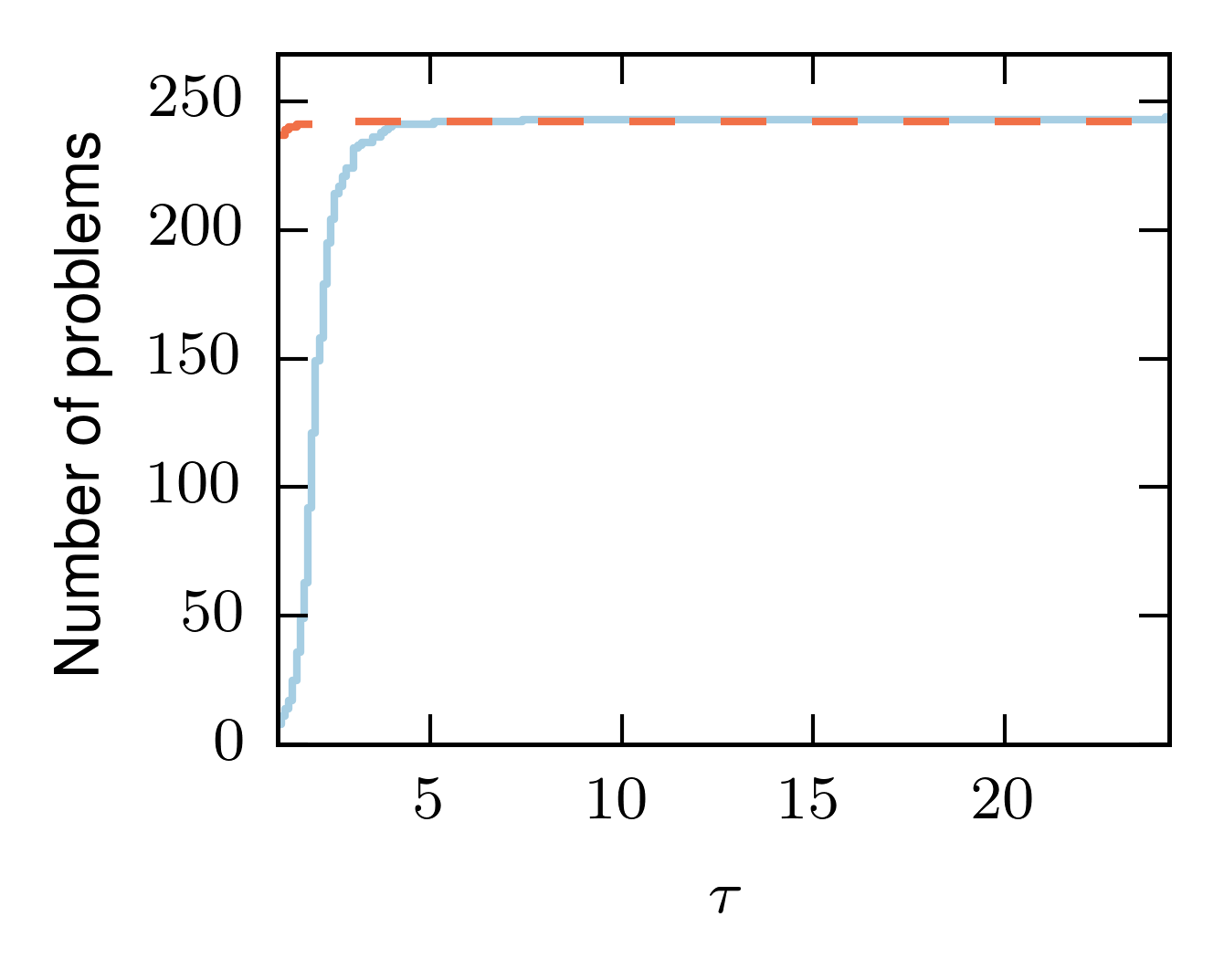} &
    \includegraphics{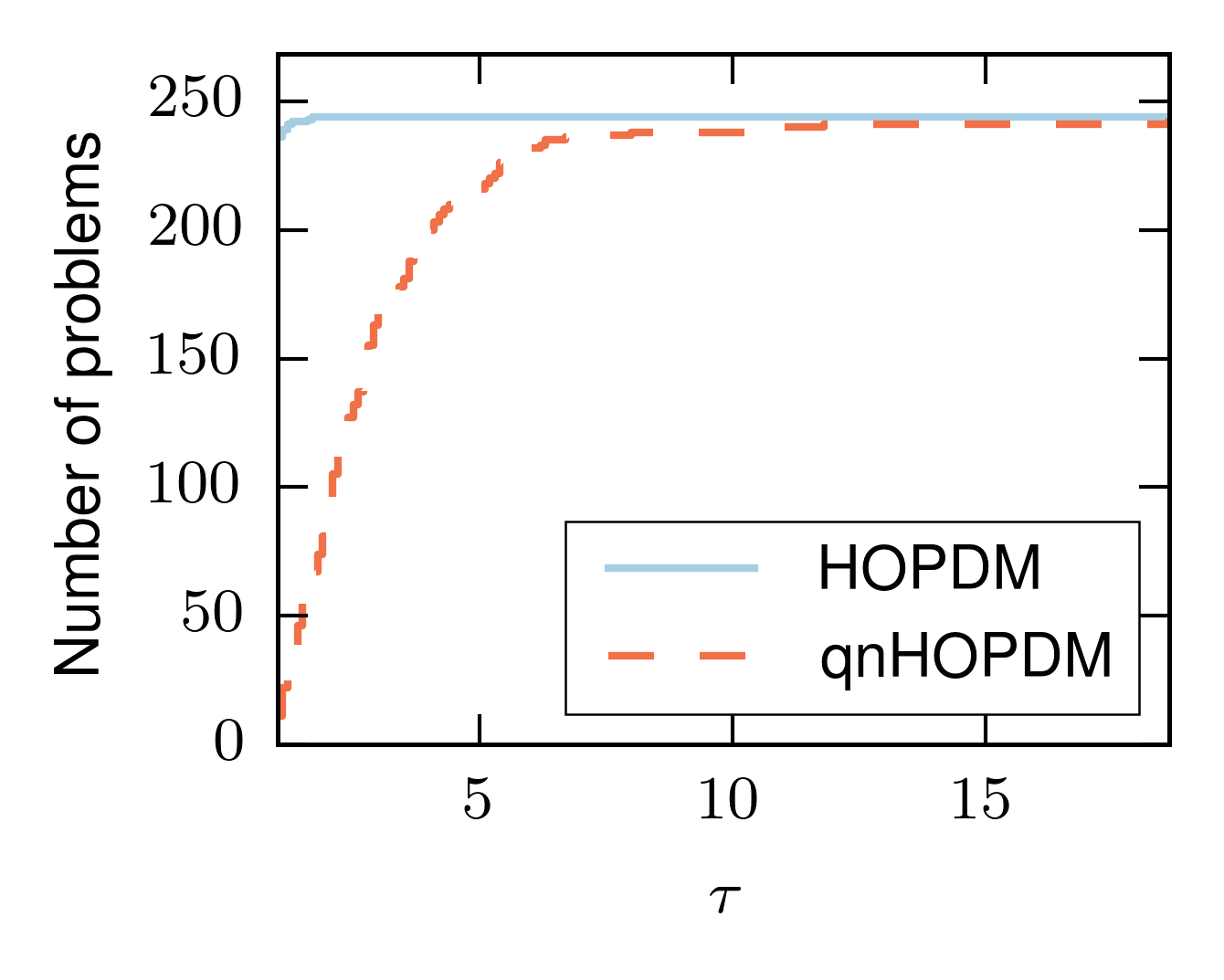} \\
    (a) &
    (b) \\ \midrule
    \multicolumn{2}{c}{With multiple centrality correctors}\\ \midrule
    Cholesky factorizations & CPU time \\
    \includegraphics{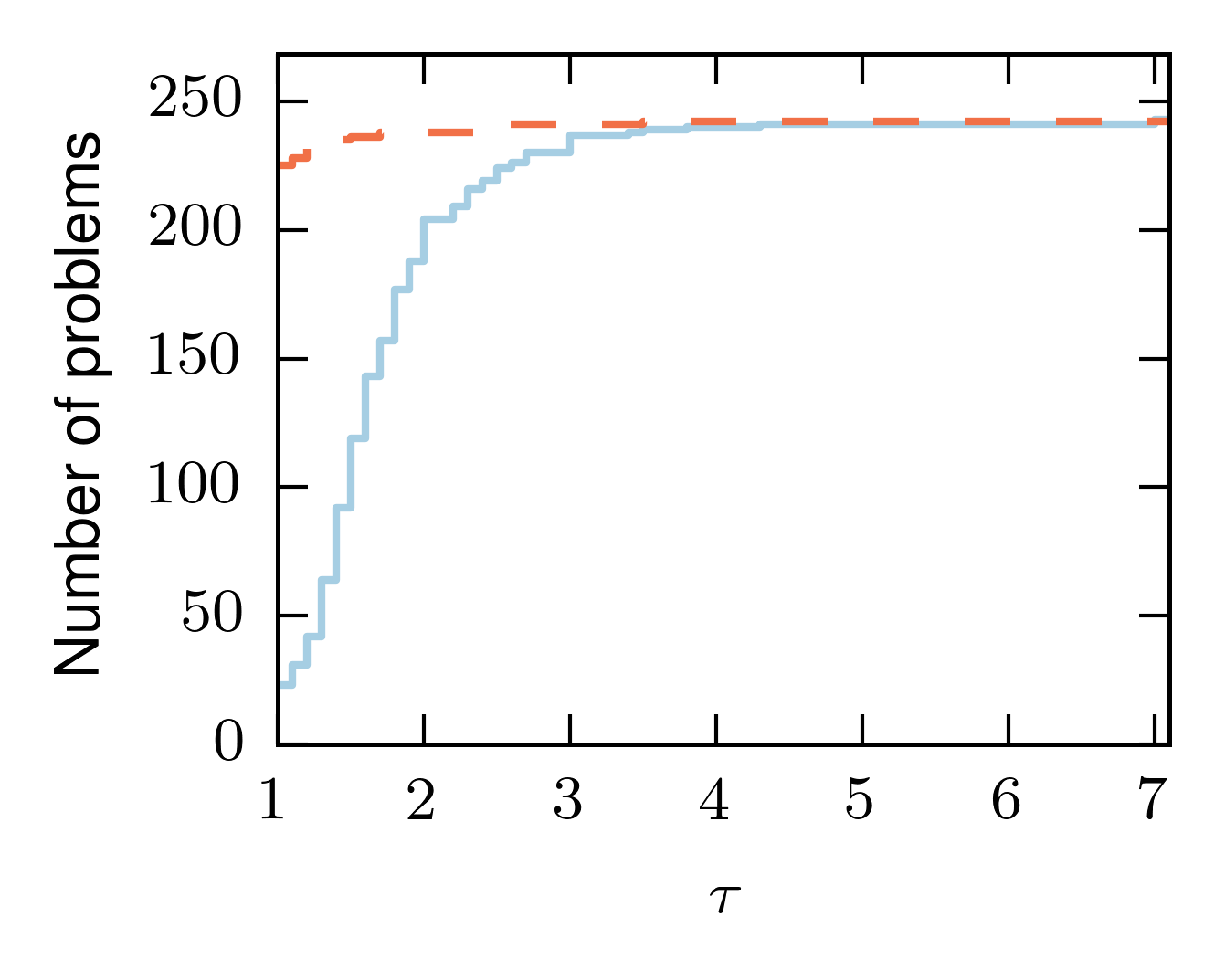} &
    \includegraphics{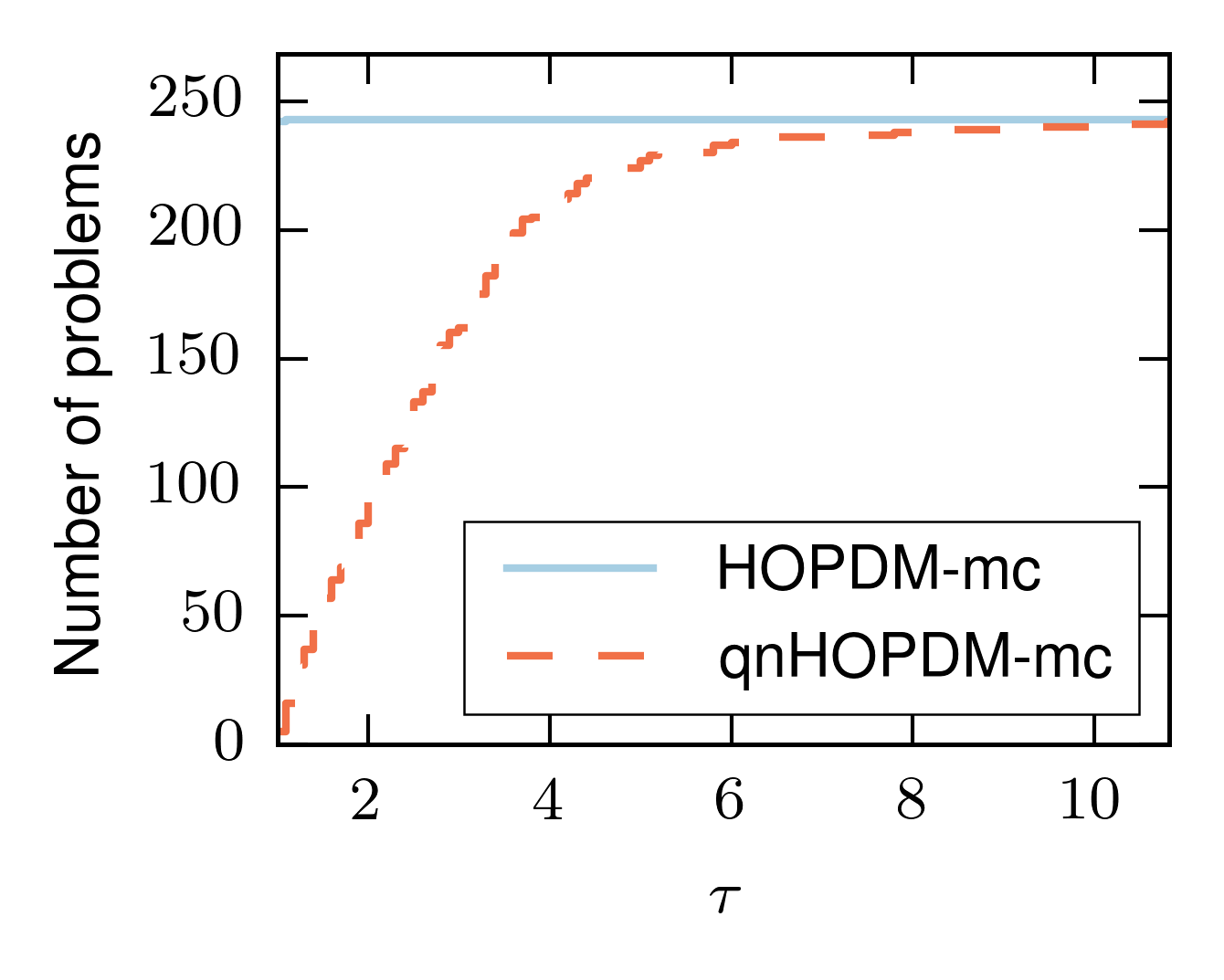} \\
    (c) &
    (d) \\
  \end{tabular}

  \caption{Performance profiles for the comparison between the
    quasi-Newton IPM and \hopdm\ without ((a) and (b)) and with ((c)
    and (d)) multiple centrality correctors for Newton steps in 244
    linear and quadratic programming problems.}
  \label{fig:pphopdm}
\end{figure}

Similarly to the previous comparison, using default parameters, 5
problems were not solved by \qnipm\ or \hopdm\ without multiple
centrality correctors, while 7 problems were not solved by \qnipmmc\
or \hopdm-mc. Criteria~\eqref{converged} was relaxed in the same way
on these problems. Using this approach, \hopdm\ is able to solve all
the 244 problems, \qnipm\ solves 242, \hopdm-mc solves 243 and
\qnipm-mc solves 242. The quasi-Newton implementations are able to
successfully reduce the number of factorizations, as shown in
Figures~\ref{fig:pphopdm}(a) and~\ref{fig:pphopdm}(c). We can see in
Figure~\ref{fig:pphopdm}(a) that from all 242 problems considered
solved by \qnipm, in 237 it uses less factorizations than \hopdm\
without multiple centrality correctors. On the other hand, for about
150 problems, \qnipm\ uses at least twice as much CPU time as \hopdm\
(Figure~\ref{fig:pphopdm}(b)). The behavior of the implementations
using multiple centrality correctors in the Newton step is similar,
but \hopdm-mc has improved efficiency results. The problems where
\qnipm\ reduces both factorizations and CPU time when compared to
\hopdm\ without centrality correctors are highlighted in
Table~\ref{tab:pphopdm}. The only problem which \qnipmmc\ uses
strictly less CPU time than \hopdm-mc is the quadratic programming
problem \texttt{cont-101}.

\begin{table}[ht!]
  \centering
\begin{tabular}{lrrrrrrrr}
\toprule
{} & \multicolumn{2}{l}{\hopdm} & \multicolumn{2}{l}{\qnipm}
   & \multicolumn{2}{l}{\hopdm-mc} & \multicolumn{2}{l}{\qnipmmc} \\
  \cmidrule(l){2-3} \cmidrule(l){4-5} \cmidrule(l){6-7} \cmidrule(l){8-9}
  {} &    \texttt{F}  & \texttt{CPUt} &    \texttt{F}  & \texttt{CPUt} &    \texttt{F}  & \texttt{CPUt}
                                      &    \texttt{F}  & \texttt{CPUt} \\
\midrule
\texttt{dfl001  } &    53 &  56.975 &      24 &  $\mathbf{36.887}$ &       24 &  28.178 &         21 &  36.122 \\
\texttt{maros-r7} &    16 &   2.362 &       8 &  $\mathbf{ 2.080}$ &       10 &   1.723 &          8 &   2.361 \\
\texttt{pilot87 } &    31 &   4.242 &      10 &  $\mathbf{ 3.277}$ &       15 &   2.472 &         11 &   3.576 \\
\texttt{cont-101} &    11 &   1.138 &       5 &  $\mathbf{ 1.090}$ &        9 &   1.255 &          5 &   $\mathbf{1.160}$ \\
\texttt{cont-200} &     9 &   6.992 &       5 &  $\mathbf{ 6.050}$ &        9 &   8.031 &         12 &  15.666 \\
\texttt{dualc8}   &   121 &   0.049 &       5 &  $\mathbf{ 0.029}$ &       61 &   0.036 &         23 &   0.056 \\
\texttt{hs35    } &     8 &   0.025 &       3 &  $\mathbf{ 0.023}$ &        7 &   0.022 &          3 &   0.023 \\
\texttt{tame    } &     5 &   0.021 &       2 &  $\mathbf{ 0.020}$ &        5 &   0.020 &          2 &   0.021 \\
\bottomrule
\end{tabular}

\caption{Problems where the quasi-Newton implementation \qnipm\ used
  strictly less CPU time than \hopdm.}
\label{tab:pphopdm}
\end{table}

Our last comparison considers 39 medium-sized problems from the QAP
collection. These problems are challenging, since they are sparse, but
their Cholesky factorization is very dense. Performance profiles were
once more used for comparing the implementations. As the algorithm
approaches the solution, the linear systems become harder to
solve. Therefore, using default \hopdm\ values for parameters
in~\eqref{converged} the number of problems solved is 21 (\hopdm), 31
(\qnipm), 25 (\hopdm-mc) and 35 (\qnipmmc). Clearly the quasi-Newton
approach benefits of using matrices that are not too close to the
solution. From the 39 problems, 19 were solved again using relaxed
parameters for the comparison between \hopdm\ and \qnipm, and 14 were
solved again for the comparison between \hopdm-mc and \qnipmmc. The
results are shown in Figure~\ref{fig:qap}. Quasi-Newton IPM is the
most efficient and robust algorithm in terms of CPU time for both
implementations, solving all 39 problems. Without multiple centrality
correctors (Figure~\ref{fig:qap}(a)), \hopdm\ has a poor performance
and is not able to solve any problem using less CPU time than
\qnipm. When multiple centrality correctors are allowed
(Figure~\ref{fig:qap}(b)), \hopdm-mc is able to solve only 10 problems
using less or equal CPU time than \qnipmmc.

Clearly, the efficiency of \qnipm\ is due to the decrease in the
number of factorizations, as shown in Table~\ref{tab:qap}. In this
table we display the number of factorizations (\texttt{F}) and CPU
time (\texttt{CPUt}) for each problem and each algorithm in all QAP
test problems considered. When no multiple centrality correctors are
allowed at Newton steps, \qnipm\ displays the biggest improvements,
being the fastest solver in all problems. The results are more
competitive when multiple centrality correctors are allowed, but
\qnipmmc\ was the most efficient in 29 problems while \hopdm-mc was the
most efficient in 10 problems.

\begin{figure}[ht!]
  \centering
  
  \begin{tabular}{cc}
    \includegraphics{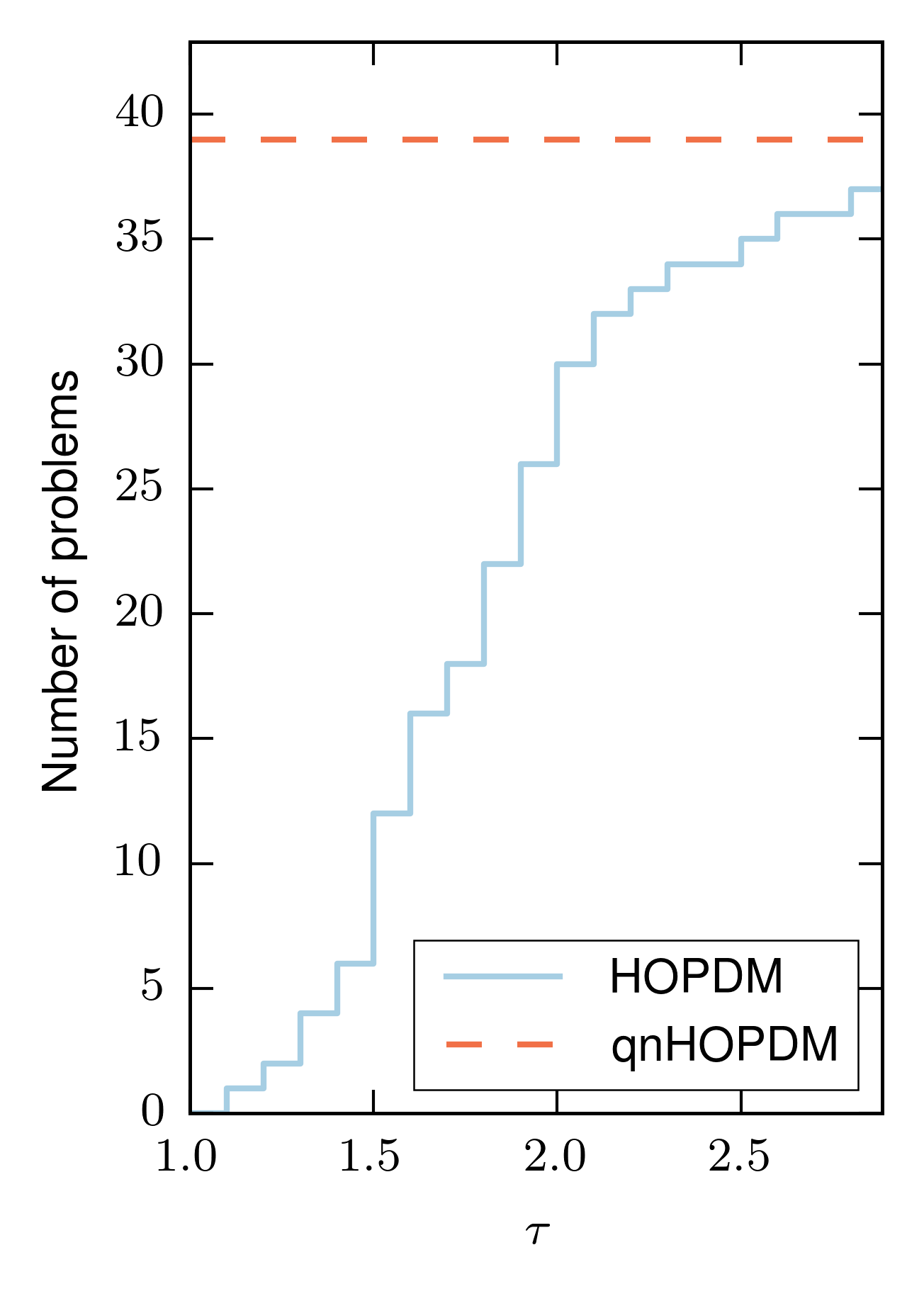} &
    \includegraphics{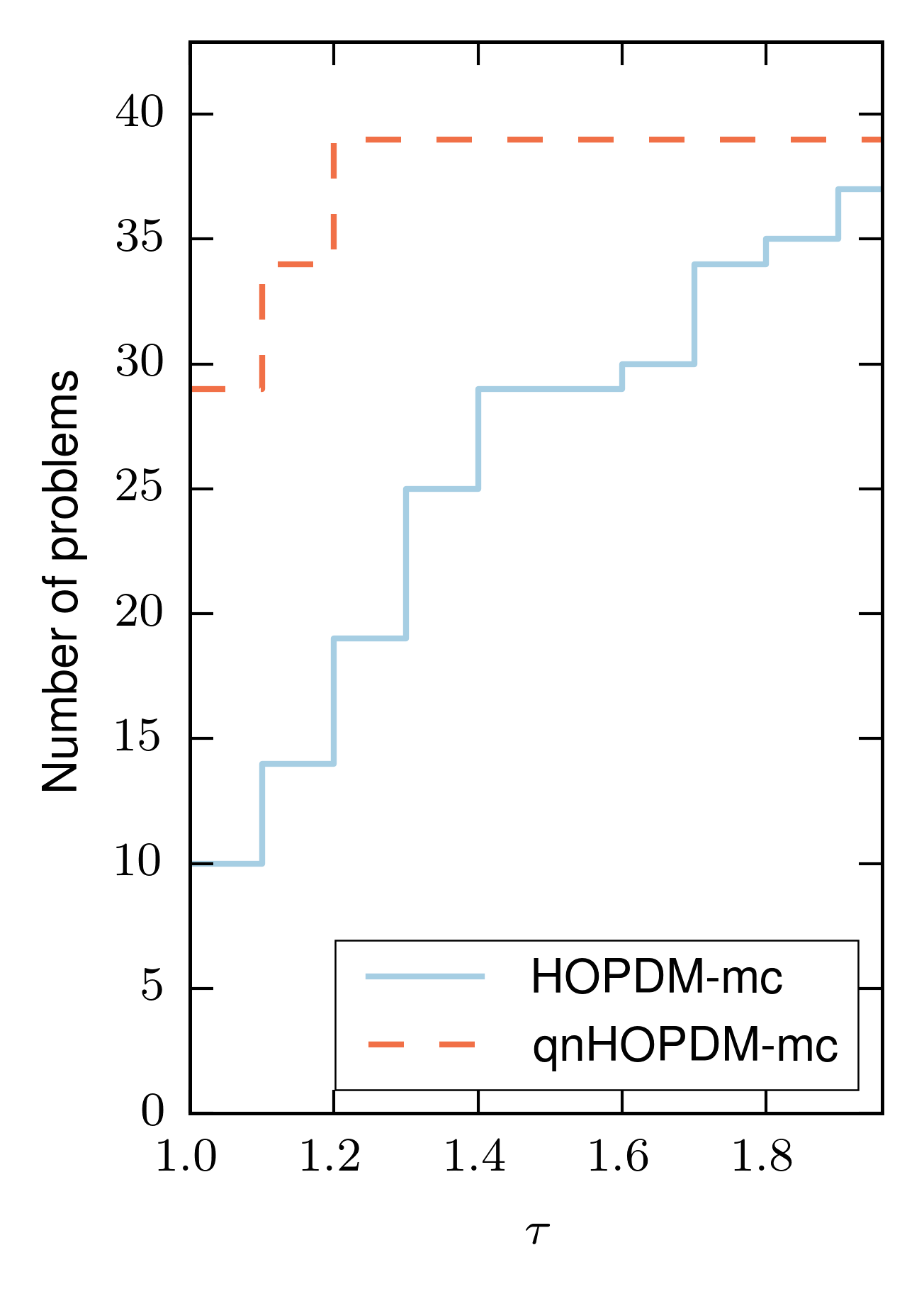} \\
    (a) &
    (b)
  \end{tabular}

  \caption{Performance profiles for the comparison between
    quasi-Newton IPM and \hopdm\ on the QAP test collection. The CPU
    time was used as performance measure.}
  \label{fig:qap}
\end{figure}

\begin{table}[ht]
  \centering
\begin{tabular}{lrrrrrrrr}
\toprule
{} & \multicolumn{2}{l}{\hopdm} & \multicolumn{2}{l}{\qnipm} & \multicolumn{2}{l}{\hopdm-mc} & \multicolumn{2}{l}{\qnipm-mc} \\
  \cmidrule(l){2-3} \cmidrule(l){4-5} \cmidrule(l){6-7} \cmidrule(l){8-9}
  {} &    \texttt{F}  & \texttt{CPUt} &    \texttt{F}  & \texttt{CPUt} &    \texttt{F}  & \texttt{CPUt}
                                      &    \texttt{F}  & \texttt{CPUt} \\
\midrule

\texttt{qap8  } &    12 &     0.657\blanksp\blanksp &       5 &    0.438\blanksp\blanksp &        9 &     0.481\blanksp\blanksp &          4 &    0.396\blanksp\blanksp \\
\texttt{qap12 } &    20 &    34.225\blanksp\blanksp &      12 &   23.052\blanksp\blanksp &       14 &    23.928\blanksp\blanksp &          9 &   19.120\blanksp\blanksp \\
\texttt{qap15 } &    15 &   199.306\notfeas\blanksp &       9 &  149.782\blanksp\blanksp &       13 &   175.021\blanksp\blanksp &         12 &  179.777\blanksp\blanksp \\
\texttt{chr12a} &    15 &    25.858\blanksp\blanksp &       6 &   14.887\blanksp\blanksp &       10 &    17.654\blanksp\blanksp &          6 &   13.563\blanksp\blanksp \\
\texttt{chr12b} &    14 &    24.127\blanksp\blanksp &       6 &   13.110\blanksp\blanksp &        9 &    16.108\blanksp\blanksp &          5 &   11.558\blanksp\blanksp \\
\texttt{chr12c} &    14 &    24.019\blanksp\blanksp &       6 &   14.297\blanksp\blanksp &       10 &    17.711\blanksp\blanksp &          6 &   14.262\blanksp\blanksp \\
\texttt{chr15a} &    28 &   365.526\notfeas\blanksp &      11 &  168.167\blanksp\blanksp &       11 &   220.737\notfeas\blanksp &         11 &  175.070\blanksp\blanksp \\
\texttt{chr15b} &    18 &   254.418\blanksp\blanksp &       9 &  138.639\blanksp\blanksp &       11 &   149.842\blanksp\blanksp &          8 &  138.697\blanksp\blanksp \\
\texttt{chr15c} &    15 &   198.142\blanksp\blanksp &       7 &  110.466\blanksp\blanksp &       10 &   137.979\blanksp\blanksp &          6 &  102.391\blanksp\blanksp \\
\texttt{chr18a} &    31 &  2267.138\notfeas\blanksp &      10 &  814.251\notfeas\blanksp &       13 &  1007.215\notfeas\blanksp &         10 &  833.171\notfeas\blanksp \\
\texttt{chr18b} &    15 &  1094.975\blanksp\blanksp &       5 &  430.104\blanksp\blanksp &       11 &   812.243\blanksp\blanksp &          5 &  436.455\blanksp\blanksp \\
\texttt{esc16a} &     9 &   233.146\blanksp\blanksp &       4 &  120.412\blanksp\blanksp &        9 &   229.835\blanksp\blanksp &          5 &  148.167\blanksp\blanksp \\
\texttt{esc16b} &     6 &   161.393\blanksp\blanksp &       3 &  128.426\notfeas\blanksp &        7 &   184.806\blanksp\blanksp &          5 &  152.563\blanksp\blanksp \\
\texttt{esc16c} &     9 &   256.499\blanksp\blanksp &       4 &  151.036\blanksp\blanksp &        6 &   168.644\notfeas\blanksp &          3 &  100.523\blanksp\blanksp \\
\texttt{esc16d} &    10 &   236.599\blanksp\blanksp &       4 &  132.912\blanksp\blanksp &        6 &   165.879\notfeas\blanksp &          4 &  126.286\blanksp\blanksp \\
\texttt{esc16e} &     9 &   228.907\blanksp\blanksp &       5 &  154.823\blanksp\blanksp &        8 &   206.985\blanksp\blanksp &          4 &  126.947\blanksp\blanksp \\
\texttt{esc16f} &     5 &   137.600\notfeas\notopti &       2 &   74.728\blanksp\blanksp &        5 &   202.329\notfeas\notopti &          2 &   78.376\blanksp\blanksp \\
\texttt{esc16g} &     7 &   184.014\notfeas\blanksp &       4 &  118.925\notfeas\blanksp &        6 &   161.090\notfeas\blanksp &          4 &  135.363\blanksp\blanksp \\
\texttt{esc16h} &     7 &   187.607\notfeas\blanksp &       4 &  124.728\blanksp\blanksp &        9 &   229.396\blanksp\blanksp &          4 &  129.765\blanksp\blanksp \\
\texttt{esc16i} &     9 &   229.359\notfeas\blanksp &       5 &  147.298\blanksp\blanksp &        8 &   210.252\notfeas\blanksp &          4 &  127.249\notfeas\blanksp \\
\texttt{esc16j} &     9 &   233.170\notfeas\notopti &       4 &  125.714\blanksp\blanksp &        8 &   190.339\notfeas\notopti &          4 &  124.838\blanksp\blanksp \\
\texttt{had12 } &    15 &    25.463\notfeas\blanksp &      13 &   23.704\blanksp\blanksp &        8 &    14.852\notfeas\blanksp &          6 &   17.055\blanksp\blanksp \\
\texttt{had14 } &    16 &   132.848\notfeas\blanksp &       6 &   53.987\blanksp\blanksp &        8 &    63.408\notfeas\blanksp &          8 &   75.801\blanksp\blanksp \\
\texttt{had16 } &    16 &   407.539\notfeas\blanksp &      13 &  370.233\notfeas\blanksp &        8 &   212.278\notfeas\blanksp &          6 &  185.092\notfeas\blanksp \\
\texttt{had18 } &    17 &  1221.709\notfeas\blanksp &      11 &  831.704\notfeas\blanksp &        8 &   636.914\notfeas\blanksp &          8 &  655.777\notfeas\blanksp \\
\texttt{nug12 } &    20 &    33.161\blanksp\blanksp &      12 &   23.069\blanksp\blanksp &       14 &    23.910\blanksp\blanksp &          9 &   21.327\blanksp\blanksp \\
\texttt{nug14 } &    17 &   129.937\notfeas\blanksp &       8 &   66.117\blanksp\blanksp &       14 &    96.694\blanksp\blanksp &         11 &   95.963\blanksp\blanksp \\
\texttt{nug15 } &    15 &   198.589\notfeas\blanksp &       9 &  141.071\blanksp\blanksp &       13 &   175.054\blanksp\blanksp &         12 &  190.730\blanksp\blanksp \\
\texttt{nug16a} &    17 &   417.437\notfeas\blanksp &      11 &  314.709\notfeas\blanksp &       16 &   391.903\blanksp\blanksp &         13 &  361.863\blanksp\blanksp \\
\texttt{nug16b} &    15 &   413.183\notfeas\blanksp &       7 &  204.477\blanksp\blanksp &       14 &   347.994\blanksp\blanksp &         11 &  301.793\blanksp\blanksp \\
\texttt{nug17 } &    17 &   732.045\notfeas\blanksp &       8 &  406.272\notfeas\blanksp &        8 &   391.035\notfeas\blanksp &          9 &  443.721\blanksp\blanksp \\
\texttt{nug18 } &    16 &  1161.936\notfeas\blanksp &       7 &  602.210\notfeas\blanksp &        9 &   921.522\notfeas\blanksp &          6 &  508.669\blanksp\blanksp \\
\texttt{rou12 } &    23 &    37.859\blanksp\blanksp &      13 &   24.526\blanksp\blanksp &       13 &    22.755\blanksp\blanksp &         10 &   23.001\blanksp\blanksp \\
\texttt{rou15 } &    23 &   296.984\blanksp\blanksp &       9 &  132.725\blanksp\blanksp &       12 &   162.789\blanksp\blanksp &          9 &  148.203\blanksp\blanksp \\
\texttt{scr12 } &    28 &    45.440\blanksp\blanksp &      11 &   21.778\blanksp\blanksp &       13 &    22.485\blanksp\blanksp &         11 &   23.858\blanksp\blanksp \\
\texttt{scr15 } &    27 &   368.647\blanksp\blanksp &      13 &  187.875\blanksp\blanksp &       16 &   212.057\blanksp\blanksp &         15 &  235.463\blanksp\blanksp \\
\texttt{tai12a} &    24 &    39.167\blanksp\blanksp &      10 &   20.823\blanksp\blanksp &       14 &    24.274\blanksp\blanksp &          8 &   20.890\blanksp\blanksp \\
\texttt{tai15a} &    24 &   324.739\blanksp\blanksp &      12 &  183.891\blanksp\blanksp &       11 &   156.699\blanksp\blanksp &         12 &  187.767\blanksp\blanksp \\
\texttt{tai17a} &    24 &  1015.653\blanksp\blanksp &      15 &  836.886\blanksp\blanksp &       12 &   528.553\blanksp\blanksp &          6 &  314.275\blanksp\blanksp \\

\bottomrule
\end{tabular}

\caption{Numerical results for the QAP collection. For each algorithm the
  number of Cholesky factorizations (\texttt{F}) and CPU time (\texttt{CPUt}) is displayed. Index
  \notfeas\ represents solutions considered not solved using default parameters while \notopti\ 
  marks solutions considered not solved using relaxed parameters.
}
\label{tab:qap}
\end{table}


\section{Conclusions}
\label{conclusions}

In this work we discussed a new approach to IPM based on rank-one
secant updates for solving quadratic programming problems. The
approach was motivated by the multiple centrality correctors, which
provide many possible points where the function $F$ can be evaluated
in order to build a good approximation of $J$. Instead of using
several points, the present approach uses only the new computed point
in order to build a low rank approximation to the unreduced matrix at
the next iteration. The computational cost of solving the quasi-Newton
linear system can be compared with the cost of computing one
corrector, as all the factorizations and preconditioners have already
been calculated.

It was shown that rank-one secant updates maintain the main structure
of the unreduced matrix. Also, several aspects of an efficient
implementation were discussed. The proposed algorithm was implemented
as a modification of algorithm \hopdm\ using the Broyden ``bad''
update, modified to preserve the sparsity structure of the unreduced
matrix. The implementation was compared with the original version of
\hopdm\ and was able to reduce the overall number of factorizations in
most of the problems. However, only in the test set containing linear
relaxations of quadratic assignment problems, the reduction in the
number of factorizations was systematically translated into the
reduction of the CPU time of the algorithm. This suggests that the
proposed algorithm is suitable for problems where the computational
cost of the factorizations is much higher than the cost of the
backsolves.

\bibliographystyle{gs}
\bibliography{gs}

\end{document}